\documentclass[11pt]{article}
\usepackage{amsmath,amssymb,amsthm}
\usepackage{url}
\usepackage{fullpage}
\usepackage{wrapfig}
\usepackage{graphicx}
\usepackage[latin1]{inputenc} 
\usepackage{hyperref}
\usepackage{xypic}

\newtheorem{theorem}{Theorem}
\newtheorem{lemma}[theorem]{Lemma}
\newtheorem{corollary}[theorem]{Corollary}

\theoremstyle{definition}
\newtheorem{remark}[theorem]{Remark}
\newtheorem{definition}[theorem]{Definition}

\newcommand\R{\ensuremath{\mathbb{R}}}
\newcommand\RG{\ensuremath{\mathbb{RG}}}
\newcommand\s{\ensuremath{\mathbb{S}}}
\newcommand\p{\ensuremath{\mathbb{P}}}
\newcommand\Q{\ensuremath{\mathbb{Q}}}
\newcommand\Z{\ensuremath{\mathbb{Z}}}

\newcommand{\SSS}{\mathcal{S}}
\newcommand{\F}{\mathcal{F}}
\newcommand{\K}{\mathcal{K}}
\newcommand{\G}{\mathcal{G}}
\newcommand{\h}{\mathcal{H}}
\newcommand{\N}{\mathcal{N}}
\newcommand{\M}{\mathcal{M}}

\newcommand{\sd}{\textrm{sd}}
\newcommand{\dD}{\dot{D}}

\makeatletter
\newcommand\homred{\ensuremath{\tilde H}}
\newcommand\inv{\ensuremath{^{-1}}}
\newcommand{\Set}[1]{\left\{ #1 \right\}}
\newcommand{\Bigbar}[1]{\mathrel{\left|\vphantom{#1}\right.\n@space}}
\newcommand{\Setbar}[2]{\Set{#1 \Bigbar{#1 #2} #2}}

\def\ifdisplay#1#2{\mathchoice{#1}{#2}{#2}{#2}}
\def\map{\@ifnextchar[{\@xmap}{\@map}
}
\def\@xmap[#1]#2#3#4#5{%
  \ifdisplay{%
    \ensuremath{#1\colon\begin{array}{rcl}#2&\to&#3\\#4&\mapsto&#5\end{array}}
  }{%
    #1\colon#2\rightarrow#3, #4\mapsto #5}
}
\def\@map#1#2#3#4{%
  \ifdisplay{%
    \ensuremath{\begin{array}{rcl}#1&\to&#2\\#3&\mapsto&#4\end{array}}
  }{%
    #1\rightarrow#2, #3\mapsto #4}
}
\makeatother
 
\title{Helly numbers of acyclic families}
\author{
\'Eric Colin de Verdi\`ere\footnote{Laboratoire d'informatique, \'Ecole
  normale sup\'erieure, CNRS, Paris, France. email: \texttt{Eric.Colin.De.Verdiere@ens.fr}}
\and
Gr\'egory Ginot\footnote{UPMC Paris VI, Universit\'e Pierre et Marie Curie, Institut
  Math\'ematique de Jussieu, Paris, France. email: \texttt{ginot@math.jussieu.fr}}
\and
Xavier Goaoc\footnote{Project-team VEGAS, INRIA, Laboratoire Lorrain de Recherche en
  Informatique et Automatique, Nancy, France. email: \texttt{xavier.goaoc@inria.fr}}
}
\date{\today}

\begin{document}
\maketitle

\begin{abstract}
  The \emph{Helly number} of a family of sets with empty intersection is the size of its largest inclusion-wise minimal sub-family with empty intersection.  Let $\F$ be a finite family of open subsets of an arbitrary locally arc-wise connected topological space~$\Gamma$. Assume that for every sub-family $\G \subseteq \F$ the intersection of the elements of $\G$ has at most $r$ connected components, each of which is a $\Q$-homology cell. We show that the Helly number of $\F$ is at most $r(d_\Gamma+1)$, where $d_\Gamma$ is the smallest integer $j$ such that every open set of~$\Gamma$ has trivial $\Q$-homology in dimension $j$ and higher. (In particular $d_{\R^d} = d$). This bound is best possible. We prove, in fact, a stronger theorem where small sub-families may have more than $r$ connected components, each possibly with nontrivial homology in low dimension. As an application, we obtain several explicit bounds on Helly numbers in  geometric transversal theory for which only ad hoc geometric proofs were  previously known; in certain cases, the bound we obtain is better than what was previously known.
\end{abstract}

\section{Introduction}

The \emph{Helly number} of a family of sets with empty intersection is the size of its largest sub-family $\F$ such that (i) the intersection of all elements of $\F$ is empty, and (ii) for any proper sub-family $\G \subsetneq \F$, the intersection of the elements of $\G$ is non-empty. This number is named after Eduard Helly, whose theorems state that any inclusion-wise minimal family with empty intersection has size at most $d+1$ if it consists of finitely many convex sets in $\R^d$~\cite{h-umkkm-23} or forms a good cover in $\R^d$~\cite{h-usvam-30}.  (For our purposes, a \emph{good cover} is a finite family of open sets where the intersection of any sub-family is empty or contractible.) In this paper, we prove \emph{Helly-type theorems} for families of non-connected sets, that is we give upper bounds on Helly numbers for such families.  (When considering the Helly number of a family of sets, we always implicitly assume that the family has empty intersection.)

\subsection{Our results}

Let $\Gamma$ be a locally arc-wise connected topological space. We let $d_\Gamma$ denote the smallest integer such that every open subset of~$\Gamma$ has trivial $\Q$-homology in dimension $d_\Gamma$ and higher; in particular, when $\Gamma$ is a $d$-dimensional manifold, we have $d_\Gamma = d$ if $\Gamma$ is non-compact or non-orientable and $d_\Gamma =d+1$ otherwise (see Lemma~\ref{lem:homdim}); for example, $d_{\R^d} =d$. We call a family $\F$ of open subsets of $\Gamma$ \emph{acyclic} if for any non-empty sub-family $\G \subseteq\F$, each connected component of the intersection of the elements of $\G$ is a $\Q$-homology cell. (Recall that, in particular, any contractible set is a homology cell.)%
\footnote{To avoid confusion, we note that an \emph{acyclic space} sometimes refers to a homology cell in the literature (see \emph{e.g.}, Farb~\cite{f-gaht-09}).  Here, the meaning is different: the intersection of a finite sub-family in an acyclic family can consist of \emph{several} $\Q$-homology cells.}  We prove the following Helly-type theorem:

\begin{theorem}\label{thm:mainresult}
Let $\F$ be a finite acyclic family of open subsets of a locally arc-wise connected topological space $\Gamma$. If any sub-family of $\F$ intersects in at most $r$ connected components, then the Helly number of $\F$ is at most $r(d_\Gamma+1)$.
\end{theorem}
\noindent We show, in fact, that the conclusion of Theorem~\ref{thm:mainresult} holds even if the intersection of small sub-families has more than $r$ connected components and has non-vanishing homology in low dimension.  To state the result precisely, we need the following definition that is a weakened version of acyclicity:
\begin{definition}
  A finite family $\F$ of subsets of a locally arc-wise connected topological space is \emph{acyclic with slack $s$} if for every non-empty sub-family~$\G\subseteq\F$ and every $i\ge\max(1,s-|\G|)$ we have $\homred_i(\bigcap_\G,\Q)=0$.
\end{definition}
\noindent
Note that, in  particular, if $s\le1$, \emph{acyclic with slack~$s$} is the same as \emph{acyclic}.  With a view toward applications in geometric transversal theory, we actually prove the following strengthening of Theorem~\ref{thm:mainresult}:
\begin{theorem}\label{thm:non-uniform}
Let $\F$ be a finite family of open subsets of a locally arc-wise connected topological space $\Gamma$. If (i) $\F$ is acyclic with slack $s$ and (ii) any sub-family of $\F$ of cardinality at least $t$ intersects in at most $r$ connected components, then the Helly number of~$\F$ is at most $r(\max(d_\Gamma,s,t)+1)$.
\end{theorem}

In both Theorems~\ref{thm:mainresult} and~\ref{thm:non-uniform} the openness condition can be replaced by a compactness condition (Corollary~\ref{coro:compact}) under an additional mild assumption. As an application of Theorem~\ref{thm:non-uniform} we obtain bounds on several \emph{transversal Helly numbers}: given a family $A_1, \ldots, A_n$ of convex sets in~$\R^d$ and letting $T_i$ denote the set of non-oriented lines intersecting $A_i$, we can obtain bound on the Helly number $h$ of $\{T_1, \ldots, T_n\}$ under certain conditions on the geometry of the $A_i$. Specifically, we obtain that $h$ is
\begin{itemize}
\item[(i)] at most $2^{d-1}(2d-1)$ when the $A_i$ are disjoint parallelotopes in $\R^d$, 
\item[(ii)] at most $10$ when the $A_i$ are disjoint translates of a convex set in $\R^2$, and
\item[(iii)] at most $4d-2$ (resp. $12$, $15$, $20$, $20$) when the $A_i$ are disjoint equal-radius balls in $\R^d$ with $d \ge 6$ (resp. $d=2$, $3$, $4$, $5$).
\end{itemize}
Although similar bounds were previously known, we note that each was obtained through an ad hoc, geometric argument.  The set of lines intersecting a convex set in $\R^d$ has the homotopy type of $\R\p^{d-1}$, and the family $T_i$ is thus only acyclic with some slack; also, the bound $4d-2$ when $d\ge 4$ in (iii) is a direct consequence of the relaxation on the condition regarding the number of connected components in the intersections of small families. Theorem~\ref{thm:non-uniform} is the appropriate type of generalization of Theorem~\ref{thm:mainresult} to obtain these results; indeed, the parameters allow for some useful flexibility (cf.\ Table~\ref{tab:parameters}, page~\pageref{tab:parameters}).

\bigskip

Our proofs builds on a technique of Kalai and
Meshulam~\cite{km-lnpthtt-08} that bounds Helly numbers from above by
arguing that certain nerves have vanishing homology in sufficiently
high dimension. More precisely, our proof of
Theorem~\ref{thm:mainresult} uses three ingredients. First, we define
the \emph{multinerve} of a family of sets as a simplicial poset that
records the intersection pattern of the family more precisely than the
usual nerve. Then, we derive from \emph{Leray's acyclic cover theorem}
a purely homological analogue of the Nerve theorem, identifying the
homology of the multinerve to that of the union of the family.
Finally, we generalize a theorem of Kalai and
Meshulam~\cite[Theorem~1.3]{km-lnpthtt-08} that relates the homology
of a simplicial complex to that of some of its projections.  Since in
this approach low-dimensional homology is not relevant, the
assumptions of Theorem~\ref{thm:mainresult} can be relaxed, yielding
Theorem~\ref{thm:non-uniform}.

\subsection{Relation to previous work}

\paragraph{Helly numbers.}
Previous bounds on Helly numbers of families of non-connected sets come in two different flavors.

 On the one hand, one can start with a ``ground'' family $\h$ whose Helly number is bounded and consider families $\F$ such that the intersection of any sub-family $\G \subseteq \F$ is a disjoint union of at most $r$ elements of $\h$. When $\h$ is closed under intersection and \emph{non-additive} (that is, the union of two disjoint elements of $\h$ is never an element of $\h$) the Helly number of $\F$ can be bounded by $r$ times the Helly number of $\h$. This was conjectured (and proven for $r=2$)  by Gr\"unbaum and Motzkin~\cite{gm-csfs-61} and a proof of the general case was recently published by Eckhoff and Nischke~\cite{en-pigeonhole}, building on ideas of Morris~\cite{m-phd-73}. Direct proofs were also given by Amenta~\cite{a-npihtt-96} in the case where $\h$ is a finite family of compact convex sets in $\R^d$ and by Kalai and Meshulam~\cite{km-lnpthtt-08} in the case where $\h$ is a good cover in $\R^d$~\cite{km-lnpthtt-08}.

On the other hand, Matou\v{s}ek~\cite{m-httucs-97} and Alon and Kalai~\cite{ak-bpn-95} showed, independently, that if $\F$ is a family of sets in $\R^d$ such that the intersection of any sub-family is the union of at most $r$ (possibly intersecting) convex sets, then the Helly number of $\F$ can be bounded from above by some function of $r$ and $d$. Matou\v{s}ek also gave a topological analogue~\cite[Theorem~2]{m-httucs-97} which is perhaps the closest predecessor of Theorem~\ref{thm:non-uniform}: he bounds from above (again, by a function of $r$ and $d$) the Helly number of families of sets in $\R^d$ assuming that the intersection of any sub-family has at most $r$ connected components, each of which is $(\lceil d/2\rceil-1 )$-connected, that is, has its $i$th homotopy group vanishing for $i \le \lceil d/2\rceil-1$.

\begin{figure}[tb]
  \centerline{\includegraphics[width=2.5cm, keepaspectratio]{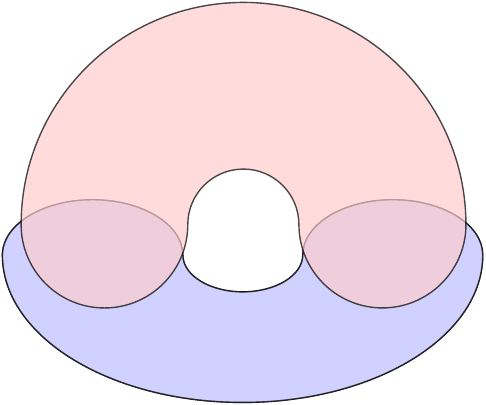}}
  \caption{A family of two sets that is not a good cover but satisfies the hypotheses of our theorem.}
  \label{fig:firstexample}
\end{figure}

Our Theorem~\ref{thm:mainresult} includes both Amenta's and Kalai-Meshulam's theorems as particular cases but is more general: Figure~\ref{fig:firstexample} shows a family for which Theorem~\ref{thm:mainresult} (as well as the topological theorem of Matou\v{s}ek) applies with $r=2$, but where the Kalai-Meshulam theorem does not (as the families of connected components is not a good cover). Our result and the Eckhoff-Morris-Nischke theorem do not seem to imply one another, but to be distinct generalizations of the Kalai-Meshulam theorem.  Theorem~\ref{thm:non-uniform} differs from Matou\v{s}ek's topological theorem on two accounts. First, his proof only gives a loose bound on the Helly number (in fact, no explicit bound is given), whereas our approach gives sharp, explicit, bounds.  Second, his theorem allows the connected components to have nontrivial \emph{homotopy} in \emph{high} dimension, whereas Theorem~\ref{thm:non-uniform} lets them have nontrivial \emph{homology} in \emph{low} dimension.

\paragraph{Nerves, Leray numbers, and \v{C}ech complexes.}
At the combinatorial level, a \emph{simplicial complex}~$X$ over a set of \emph{vertices} $V$ is a non-empty family of subsets of $V$ closed under taking subsets; in particular, $\emptyset$ belongs to every simplicial complex. An element $\sigma$ of~$X$ is a \emph{simplex}; its \emph{dimension} is the cardinality of $\sigma$ minus one; a $d$-simplex is a simplex of dimension~$d$. The \emph{nerve} of a (finite) family $\F$ of sets is the simplicial complex
\[ \N(\F) = \Setbar{H \subseteq \F}{\bigcap_{\alpha \in H} \alpha \neq
\emptyset}\]
with vertex set $\F$. The \emph{Nerve theorem} of Borsuk~\cite{b-nfhg-03,b-iscsc-48} asserts that if $\F$ is a good cover, then its nerve adequately captures the topology of the union of the members of $\F$; namely, $\N(\F)$ has the same homology groups (in fact, the same homotopy type) as $\bigcup_{\alpha \in \F} \alpha$.

That the Helly number of a good cover in $\R^d$ is at most $d+1$ can be easily derived from the Nerve theorem. Indeed, if an inclusion-wise minimal family with empty intersection has $k$ elements, its nerve is the boundary of a $(k-1)$-simplex and has therefore nontrivial homology in dimension $k-2$. The bound on the Helly number of good covers then follows from the simple observation that any open subset of $\R^d$ has trivial homology in dimension $d$ or larger. The \emph{Leray number} $L(X)$ of a simplicial complex $X$ with vertex set $V$ is defined as the smallest integer $j$ such that for any $S\subseteq V$ and any $i \ge j$ the reduced homology group $\homred_i(X[S],\Q)$ is trivial. (Recall that $X[S]$ is the sub-complex of $X$ \emph{induced} by $S$, that is, the set of simplices of $X$ contained in $S$). Using this notion, the above observation simply states that the Helly number of a collection of sets exceeds the Leray number of its nerve by at most one.

The homology of a union of a family of sets that do not form a good cover may not be captured by the nerve of that family. Indeed, in the example of Figure~\ref{fig:firstexample} the union has the homotopy type of a circle but the nerve is contractible. When the family is acyclic, one can nevertheless relate, via Leray's acyclic cover theorem\footnote{A central result in (co)sheaf (co)homology. We will, however, mainly be interested in the case of singular (co)homology, that is the case of a constant sheaf.} (see~\cite{BoTu, KaSc, Br-sheaf-97, Go-taf-73, f-gaht-09} for instance) the (co)homology of the union to the (co)homology of the \emph{\v{C}ech complex} of the cosheaf given by the connected components of the various intersections, a more complicated algebraic structure than the nerve. The notion of multinerve that we introduce can be interpreted as a \v{C}ech complex (with a constant sheaf), and therefore Leray's acyclic cover theorem applies; but this notion also retains the combinatorial and simplicial flavor of the nerve.


\paragraph{Discrete and computational geometry.}
The Helly number of a family of sets, like its Vapnik-Chervonenkis dimension, captures elegantly some aspects of its underlying discrete structure and, as such, received considerable attention from discrete geometers~\cite{dgk-htr-63,e-hrctt-93}. Helly numbers also arise naturally as size of \emph{basis} for certain geometric optimization problems~\cite{a-httgl-94}, and are therefore also of interest to computational geometers.  For instance, the radius of the smallest cylinder enclosing a set of points in~$\R^d$ is the smallest value of~$r$ such that there exists a line piercing all the balls of radius~$r$ centered at these points.  If $r$ is less than half the smallest inter-point distance, the balls are disjoint and the Helly number of the family of transversals of these balls is thus bounded; as a consequence, standard algorithms for LP-type problems~\cite{a-httgl-94} can compute this value of $r$ in $O(n)$ time.  In contrast, if $r$ is larger or equal to the smallest inter-point distance, the balls can intersect and the Helly number may not be bounded; the problem then becomes much harder, and already for points in $\R^3$ the best known algorithm takes $O(n^{3+\epsilon})$ time~\cite{aas-ltbsectd-99}.

The study of Helly number of sets of lines (or more generally, $k$-flats) intersecting a collection of subsets of $\R^d$ developed into a sub-area of discrete geometry known as geometric transversal theory~\cite{w-httgt-04}. The bounds (i)--(iii) implied by Theorem~\ref{thm:non-uniform} were already known in some form. Specifically, the case (i) of parallelotopes is a theorem of Santal\'o~\cite{s-tscpap-40}, the case (ii) of disjoint translates of a convex figure was proven by Tverberg~\cite{t-pgcct-89} with the shaper constant of $5$ and the case (iii) of disjoint equal-radius balls was proven with the constant $4d-1$ by Cheong et al.~\cite{cghp-hhtdus-08}. Each of these theorems was, however, proven through ad hoc arguments and it is interesting that Theorem~\ref{thm:non-uniform} traces them back to the same principles: controlling the homology and number of the connected components of the intersections of all sub-families. 

\subsection{Proof outline}

Our proof of Theorem~\ref{thm:mainresult} extends the key ingredient of the proof by Kalai and Meshulam~\cite{km-lnpthtt-08} of the following result: Let $\F$ be a good cover in $\R^d$ and $\G$ be a family such that the intersection of every sub-family of $\G$ has at most $r$ connected components, each of which is a member of $\F$; then the Helly number of~$\F$ is at most $r(d+1)$.  Their proof can be summarized as follows.  Let $\h$ denote the family of connected components of elements of $\G$, counted with multiplicity: each set appears as many times in $\h$ as there are elements in $\G$ that have it as a connected component. Now, consider the projection $\h \rightarrow \G$ that maps each element of $\h$ to the element of $\G$ having it as a connected component. This projection extends to a map $\N(\h) \rightarrow \N(\G)$ that is at most $r$-to-one and preserves the dimension (that is, maps a $k$-simplex to a $k$-simplex).  This turns out to imply that $L(\N(\G))$ is at most $rL(\N(\h))+r-1$ (Theorem~1.3 of~\cite{km-lnpthtt-08}, a statement we refer to as the ``projection theorem''). Since every element of~$\h$ belongs to~$\F$, the multiset~$\h$ is also a good cover in~$\R^d$; the Nerve theorem implies that $L(\N(\h))$ is at most $d$ and an upper bound of $r(d+1)$ on the Helly number of $\G$ follows.

In our setting, we assume that $\G$ is such that the intersection of any sub-family has at most $r$ connected components, each of which is a $\Q$-homology cell; note that this condition holds in the setting of Kalai-Meshulam but that the reciprocal is not always true (see the example of Figure~\ref{fig:firstexample}). In particular, the family $\h$ of connected components of members of $\G$ need not be a good cover and the Nerve theorem no longer bounds $L(\N(\h))$. We address this issue by introducing a variant of the nerve where each sub-family of $\G$ defines a number of simplices equal to the number of connected components in its intersection; we call this ``nerve with multiplicity'' the \emph{multinerve} and encode it as a simplicial poset. For the families that we consider, a ``homology multinerve theorem'' (Theorem~\ref{thm:multinerf}), stating that the multinerve captures the homology of the union, follows from standard arguments in algebraic topology. We then generalize the projection theorem of Kalai and Meshulam to any dimension-preserving projection from a simplicial poset onto a simplicial complex (Theorem~\ref{thm:new-proj}); Theorem~\ref{thm:mainresult} follows.

\paragraph{Organization.} 
After a quick overview of simplicial posets and the definition of the multinerve (Section~\ref{sec:posets}), we prove our multinerve theorem that relates the homology of the multinerve of an acyclic family (possibly with some slack) to the homology of the union of the elements of that family (Section~\ref{sec:thm-multinerf}). We then move on to generalize the projection theorem of Kalai and Meshulam (Section~\ref{sec:projection}) before proving Theorems~\ref{thm:mainresult} and~\ref{thm:non-uniform} (Section~\ref{sec:proof}). Finally, we explore some applications to geometric transversal theory (Section~\ref{sec:transversal}).

\section{Simplicial posets and multinerves}
\label{sec:posets}

In this section, we recall useful properties of simplicial posets, which are a generalization of simplicial complexes, and introduce the \emph{multinerve}, a simplicial poset that generalizes the notion of nerve.

\subsection{Preliminaries on simplicial posets}

For any finite set~$X$, we denote by $|X|$ its cardinality and by $2^X$ the family of all subsets of $X$ (including the empty set and $X$ itself). We abbreviate $\bigcap_{t \in A}t$ in $\bigcap_A$ and $\bigcup_{t \in A}t$ in $\bigcup_A$. We now describe how various properties of simplicial complexes can be generalized to simplicial posets; for more thorough discussions of these objects we refer to the book of Matou\v{s}ek~\cite[Chapter 1]{m-ubut-03} for simplicial complexes and to the paper of Bj\"orner~\cite{b-posets84} for simplicial posets.

\paragraph{Definition.}
Intuitively, a simplicial partially ordered set (simplicial poset for short) is a set of simplices with an incidence relation; a $d$-simplex still has $d+1$ distinct vertices; however, in contrast to simplicial complexes, there may be several simplices with the same vertex set, but no two can be incident to the same higher-dimensional simplex.

Formally, let $X$ be a finite set and $\preceq$ a partial order on~$X$ (we also say that $(X,\preceq)$ is a poset, or partially ordered set). Let $[\alpha,\beta] = \{ \tau \in X \mid \alpha \preceq \tau \preceq \beta\}$ denote the \emph{segment} defined by $\alpha$ and $\beta$ (similarly, $[\alpha,\beta)$, $(\alpha,\beta]$, and $(\alpha, \beta)$ denote the segments where one or both extreme elements are omitted). We say that $X$ is \emph{simplicial} if (i) it admits a \emph{least element}~$0$, that is $0 \preceq \sigma$ for any $\sigma \in X$, and (ii) for any $\sigma \in X$, there is some integer~$d$ such that the lower segment $[0,\sigma]$ is isomorphic to the poset of faces of a $d$-simplex, that is, $2^{\{0,\ldots,d\}}$ partially ordered by the inclusion; $d$ is the \emph{dimension} of $\sigma$. The elements of~$X$ are called its \emph{simplices}. We call \emph{vertices} the simplices of dimension $0$ and we say that $\tau$ is \emph{contained in} (or \emph{a face of}) $\sigma$ if $\tau \preceq \sigma$.

\begin{remark}\label{rk:vertices}
Let $\tau$ be a simplex of a simplicial poset with set of vertices $V$. The map that associates to any face of $\tau$ the set of vertices of that face is a bijection between $[0,\tau]$ and $2^V$. There may, however, exist several simplices with the same set of vertices, but no two of them can be faces of one and the same simplex.
\end{remark}

The simplices of a simplicial complex, ordered by inclusion, form a simplicial poset (with $\emptyset$ as least element), but the converse is not always true.  For example, the one-dimensional simplicial complexes are precisely the graphs without loops or multiple edges, while the one-dimensional simplicial posets correspond to the graphs without loops (but possibly with multiple edges).

If $X$ is a simplicial poset with vertex set $V$ and $S \subseteq V$, the \emph{induced simplicial sub-poset} $X[S]$ is the poset of elements of $X$ whose vertices are in $S$, ordered by the order of $X$.  A map $\varphi:X\to Y$ between two simplicial posets $X$ and~$Y$ is \emph{monotone} if it preserves the order, that is for any $\sigma, \tau \in X$ $\sigma\preceq\tau\Rightarrow \varphi(\sigma)\preceq\varphi(\tau)$, \emph{dimension-preserving} if for any $\sigma \in X$ the dimension of $\varphi(\sigma)$ equals the dimension of $\sigma$, and \emph{at most $r$-to-one} if for any $\sigma \in Y$ the set $\varphi^{-1}(\sigma)$ has cardinality at most $r$.

For future reference we state the following easy lemma.
\begin{lemma}\label{lem:dimpres}
If $f: X \to Y$ is a monotone, dimension-preserving map between two simplicial posets $X$ and $Y$ then for any $\sigma \in X$, $f$ is a bijection from $[0,\sigma]$ onto $[f(0), f(\sigma)]$.
\end{lemma}
\begin{proof}
  Let $\sigma \in X$ and let $V$ denote the set of vertices of $\sigma$. By definition, $[0,\sigma]$ and $[f(0),f(\sigma)]$ both have cardinality $2^d$ where $d$ is the dimension of~$\sigma$, so it suffices to prove that $f$ is one-to-one.  The map $\pi$ associating any $\tau \in [0,\sigma]$ to its set of vertices is a bijection from $[0,\sigma]$ onto $2^V$ by Remark~\ref{rk:vertices}. Now, assume for the sake of a contradiction that there exist two distinct elements $\tau_1, \tau_2 \in [0,\sigma]$ with $f(\tau_1)=f(\tau_2)$. Since the set of images of the vertices of $\tau_i$ is the set of vertices of $f(\tau_i)$, by monotonicity, there must be two distinct vertices $v_1, v_2 \in V$ with $f(v_1)=f(v_2)$. Now, $\pi^{-1}(\{v_1,v_2\})$ is a $1$-simplex of $X$ whose image by~$f$ is $\{f(v_1)\}$, a $0$-simplex of $Y$, contradicting the assumption that $f$ is dimension-preserving.
\end{proof}

\paragraph{Simplicial posets as simplicial sets.}
Simplicial posets lie in between simplicial complexes and the more general notion of simplicial sets as used in algebraic topology~\cite{May-sSet-67, GoJa-sHT-99}. In particular, a simplicial poset comes with natural face operators and determines a simplicial set in the same way as a simplicial complex determines a simplicial set (see~\cite{May-sSet-67} for instance). Indeed, the simplicial set $X_\bullet$ associated to a simplicial poset $X$ is the simplicial set obtained by adding all degeneracies of the simplices of $X$. In other words, $X$ is the set of non-degenerate simplices of $X_\bullet$. In this paper, we will not have to deal with the degeneracies so that the reader can forget them. The notion of monotone, dimension-preserving maps corresponds to morphisms of simplicial \emph{sets}, whereas the notion of monotone maps between simplicial posets extends the notion of simplicial maps for simplicial \emph{complexes}.

\paragraph{Geometric realization.}
To every simplicial poset~$X$, we associate a topological space~$|X|$, its
\emph{realization}, where each $d$-simplex of~$X$ corresponds to a
geometric $d$-simplex (by definition, a geometric $(-1)$-simplex is empty).
We build up the realization of~$X$ by increasing dimension.  First, create
a single point for every vertex (simplex of dimension~$0$) of~$X$.  Then,
assuming all the simplices of dimension up to~$d-1$ have been realized,
consider a $d$-simplex~$\sigma$ of~$X$.  The open lower
interval~$[0,\sigma)$ is isomorphic to the boundary of the $d$-simplex by
definition; we simply glue a geometric $d$-simplex to the realization of
that boundary. It is clear that the geometric realization of a simplicial
complex viewed as a simplicial poset as above coincides with the usual
notion for simplicial complexes. Further, this geometric realization also
coincides with (meaning is homeomorphic to) the geometric realization of
the associated simplicial set (see~\cite{May-sSet-67}).

\paragraph{Barycentric subdivision.}
Recall that to any (not necessarily simplicial) poset $(P,\preceq)$ is
associated a simplicial complex $\Delta(P)$ called the \emph{order complex}
of~$P$---the vertices of $\Delta(P)$ are the elements of $P$ and its
$d$-simplices are the totally ordered subsets of $P$ of size $d+1$ (also
called its \emph{chains}).

\begin{figure}[tb]
  \centerline{\includegraphics[width=14cm,
    keepaspectratio]{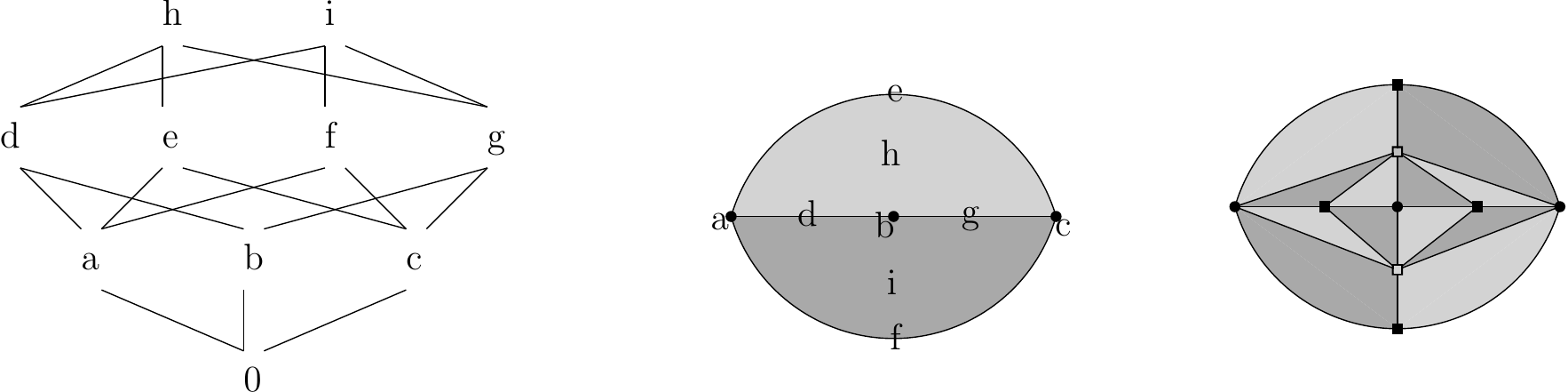}}
  \caption{Left: (Hasse diagram of) a simplicial poset $X$. Center: a
    geometric realization of $X$. Right: the natural geometric
    realization of $\sd(X)$.}
  \label{fig:exemple}
\end{figure}

The \emph{barycentric subdivision} $\sd(X)$ of a simplicial poset~$X$ with least element $0$ is defined to be $\Delta(X\setminus\Set{0})$, the order complex of $X\setminus\Set{0}$. The vertices of $\sd(X)$ are the non-empty simplices of $X$ and every chain of $d$ faces of distinct dimension contained in one another form a $(d-1)$-simplex of $\sd(X)$.  This generalizes the barycentric subdivision for simplicial complexes.

As for simplicial complexes, a geometric realization of $\sd(X)$ can be obtained from a subdivision of the geometric realization of $X$, as follows (see Figure~\ref{fig:exemple}).  The barycentric subdivision of a 0-simplicial poset (which is also a simplicial complex) is itself.  Let $d\geq1$; assume that the $(d-1)$-skeleton of~$X$ (the simplicial sub-poset of~$X$ obtained by keeping only its simplices of dimension at most $d-1$) has already been subdivided.  We now explain how to subdivide a $d$-simplex~$\sigma$ of~$X$.  Let $v$ be a new vertex in the interior of the geometric realization of~$\sigma$.  The $(d-1)$-simplices on the boundary of~$\sigma$ have already been subdivided; let $B_\sigma$ be the set of these subdivided $(d-1)$-simplices.  For every $(d-1)$-simplex~$\tau$ in~$B_\sigma$, we insert in~$\sd(X)$ the $d$-simplex whose vertices are $v$ and those of~$\tau$.  Together, these simplices form a subdivision of~$\sigma$.  By induction, every $d$-simplex of~$X$ is subdivided into $(d+1)!$ $d$-simplices.  In particular, the geometric realization of $X$ (as a simplicial poset) is homeomorphic to the geometric realization of $\sd(X)$ (as a simplicial complex).

\paragraph{Homology.}
The \emph{homology} of a simplicial poset can be defined in three different ways: as a direct extension of simplicial homology for simplicial complexes, as a special case of simplicial homology of simplicial sets~\cite{May-sSet-67, GoJa-sHT-99}, or via the singular homology of its geometric realization; all three definitions are equivalent in that they lead to isomorphic homology groups.

We emphasize that, in this paper, we only consider homology over~$\Q$. We will use, in  Section~\ref{sec:thm-multinerf}, both the singular homology viewpoint (in particular that a simplicial poset and its barycentric subdivision have isomorphic homology groups) and the simplicial viewpoint, where the homology is defined via chain complexes.  For the reader's convenience, we now quickly recall this latter definition.

Let $X$ be a simplicial poset and assume chosen an ordering on the set of vertices of $X$.  For $n\ge 0$, let $C_n(X)$ be the $\mathbb{Q}$-vector space with basis the set of simplices of $X$ of dimension exactly $n$.  If $\sigma$ is an $n$-dimensional simplex, the lower segment $[0,\sigma]$ is isomorphic to the poset of faces of a standard $n$-simplex $2^{\{0,\dots,n\}}$; here we choose the isomorphism so that it preserves the ordering on the vertices. Thus, we get $n+1$ faces $d_i(\sigma) \in X $ (for $i=0,\dots,n$), each of dimension $n-1$: namely, $d_i(\sigma)$ is the (unique) face of~$\sigma$ whose vertex set is mapped to $\{0,\dots,n\}\setminus\{i\}$ by the above isomorphism.  Extending the maps $d_i$ by linearity, we get the \emph{face operators} $d_i: C_n(X)\to C_{n-1}(X)$. Let $d=\sum_{i=0}^n (-1)^id_i$ be the linear map $C_n(X)\to C_{n-1}(X)$ (which is defined for any $n\ge 0$). The fact that $d\circ d=0$ is easy and follows from the same argument as for simplicial complexes since it is computed inside the vector space generated by $[0,\sigma]$ which is isomorphic to a standard simplex.  The (simplicial) $n$th homology group $H_n(C_\bullet(X),d)$ is defined as the quotient vector space of the kernel of $d:C_n(X)\to C_{n-1}(X)$ by the image of $d:C_{n+1}(X)\to C_n(X)$.
  
\begin{remark}
  Let $X_\bullet$ be the simplicial set associated to $X$.  By construction, $C_n(X)$ is isomorphic to the normalized chain complex $N_n(X_\bullet)$ of $X_\bullet$ (see~\cite{May-sSet-67, GoJa-sHT-99, Weibel}) which is the quotient vector space of $\mathbb{Q}(X_\bullet)$ by the subspace spanned by the degenerate simplices. It is a standard fact that the normalized chain complex has the same homology as the simplicial set (see~\cite{May-sSet-67, GoJa-sHT-99, Weibel}).
\end{remark}
  
\smallskip

In the sequel, we denote by $H_i(O)$ the $i$th $\Q$-homology group of~$O$ (whether $O$ is a simplicial poset, its associated geometric realization, or a topological space), and by $\homred_i(O)$ the corresponding \emph{reduced} homology group~\cite{h-at-02}.\footnote{We use the convention that the reduced homology of the empty set is trivial except in dimension~$-1$, where it is~$\Q$.  In particular the definition of the Leray number, given in the introduction, makes implicitly use of this convention.}
Along with the notion of induced simplicial sub-poset and homology groups, the notion of \emph{Leray number} extends immediately to simplicial posets.

\paragraph{Links.}
Let $X$ be a simplicial complex, and let $\sigma$ be a simplex of~$X$.  The
\emph{link} of~$\sigma$ in~$X$ is the sub-complex of~$X$ defined as
\[\Setbar{\tau\in X}{\tau\cap\sigma=\emptyset,\ \tau\cup\sigma\in X}.\]
This is a standard notion in combinatorial topology; a nice topological
feature of the link of~$\sigma$ is that it has the same homotopy type as a
neighborhood of~$\sigma$ minus~$\sigma$ itself in the realization of~$X$.

This notion can be extended to simplicial posets: the link of~$\sigma$ in a
simplicial poset~$X$ would be the set of simplices~$\tau$ disjoint
from~$\sigma$ and such that~$\sigma$ and of~$\tau$ are all contained in at
least one simplex of~$X$.  However, it is not hard to prove that the above
topological property does not always hold for simplicial posets.  For
example, consider the simplicial complex made of two vertices and two edges
connecting them.

Instead, we will work on the barycentric subdivision of~$X$.  Given $\sigma \in X$, we denote by $D_X(\sigma)$ the order complex of $[\sigma, \cdot]$ and by $\dot D_X(\sigma)$ the order complex of $(\sigma, \cdot]$.  Both are sub-complexes of $\sd(X)$.  We will use the fact that $D_X(\sigma)$ retracts to~$\sigma$ and is therefore contractible.  Kalai and Meshulam~\cite{km-lnpthtt-08} use that when $X$ is a simplicial complex, $\dot D_X(\sigma)$ is isomorphic to the barycentric subdivision of the link of~$\sigma$ in~$X$.  This property is, again, false for simplicial posets; in our proof we have to find a way to avoid using the link.

\subsection{Multinerves}\label{sec:def-multinerf}

\begin{figure}[tb]
  \centerline{\includegraphics[width=14cm,
    keepaspectratio]{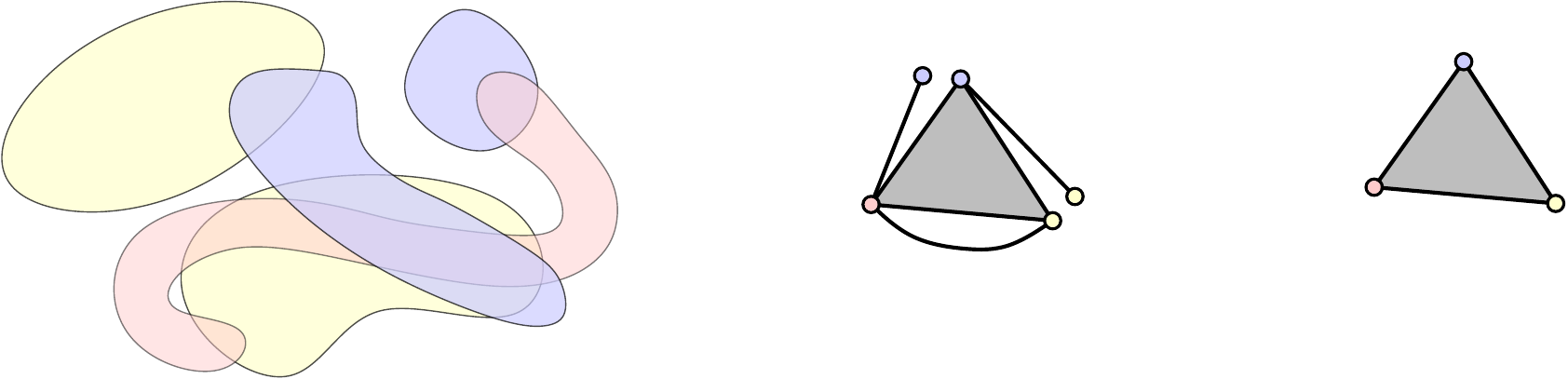}}
  \caption{Left: A family~$\F$ of subsets of~$\R^2$.  Middle: Its multinerve~$\M(\F)$.  Right: Its nerve~$\N(\F)$.}
  \label{fig:Exemple-multinerf}
\end{figure}

We define the \emph{multinerve} $\M(\F)$ of a finite family $\F$
of subsets of a topological space as the set
\[ \M(\F) = \Setbar{(C,A)}{A \subseteq \F \hbox{ and } C \hbox{ is a
    connected component of }\bigcap\nolimits_A}.\]
By convention, we put $\bigcap_\emptyset = \bigcup_\F$, and in particular,
$(\bigcup_\F,\emptyset)$ belongs to~$\M(\F)$. We turn $\M(\F)$ into a poset
by equipping it with the partial order\footnote{To get an intuition,
  it does not harm to assume that, whenever $A$ and~$A'$ are different
  subsets of~$\F$, the connected components of~$\bigcap_A$ and
  of~$\bigcap_{A'}$ are different.  Under this assumption, $\M(\F)$
  can be identified with the set of all connected components of the
  intersection of any sub-family of~$\F$, equipped with the opposite of
  the inclusion order.}
\[(C',A') \preceq (C,A) \Leftrightarrow C \supseteq C' \hbox{ and } A
\subseteq A'.\]
See Figure~\ref{fig:Exemple-multinerf} for an example.  Intuitively, $\M(\F)$ is an ``expanded'' version of~$\N(\F)$: while $\N(\F)$ has one simplex for each non-empty intersecting sub-family, $\M(\F)$ has one simplex for each \emph{connected component} of an intersecting sub-family. More precisely, the image of $\M(\F)$ through the projection on the second coordinate $\pi:(C,A) \mapsto A$ is the nerve $\N(\F)$; for any $A \in \N(\F)$, the cardinality of $\pi\inv(A)$ is precisely the number of connected components of $\bigcap_A$.

\begin{lemma}\label{L:multinervesimplicial}
The poset $\M(\F)$ is simplicial.
\end{lemma}
\begin{proof}
  The projection on the second coordinate identifies any lower segment
  $[(\bigcup_\F, \emptyset), (C,A)]$ with the simplex $2^A$. Indeed,
  let $A'\subseteq A$ and let $C'\subseteq\bigcup_\F$.  The lower
    segment $[(\bigcup_\F, \emptyset), (C,A)]$ contains $(C',A')$ if and
    only if $C'$ is the connected component of~$\bigcap_{A'}$
  containing~$C$.  Moreover, by definition, $\M(\F)$ contains a least
  element, namely $(\bigcup_\F, \emptyset)$.  The statement follows.
\end{proof}

A classical theorem of Leray~\cite{Go-taf-73, Br-sheaf-97, Spa, KaSc} states that the \v{C}ech complex of an acyclic
cover captures the homology of its union. Multinerves (or more precisely their chain complex) can be
  interpreted as a \v{C}ech complex (Section~\ref{sec:cech}) and thus afford to use
   Leray's theorem (Theorem~\ref{thm:multinerf}) from which an
  upper bound on the Leray number of the multinerve of an acyclic cover
  follows (Corollary~\ref{cor:multinerf}). In our context, we also give a slightly stronger
  variant of Leray's theorem and the corresponding bound on the Leray
  number in the case where intersections of few elements of the cover are
  allowed to have non-zero homology in low dimension; this variant will be
  necessary for our applications to geometric transversal theory.

\section{Homological multinerve theorem}
\label{sec:thm-multinerf}

In this section, we prove the following generalization of the Nerve theorem:

\begin{theorem}[Homological Multinerve Theorem]\label{thm:multinerf}
  Let $\F$ be a  family of open sets in a locally arc-wise connected topological space $\Gamma$.  If $\F$ is \emph{acyclic with slack~$s$} then $\homred_\ell(\M(\F))\cong\homred_\ell(\bigcup_\F)$ for any non-negative integer~$\ell\geq s$.  
\end{theorem}
 
\noindent
Let us emphasize that the case $s=0$ corresponds to the usual Nerve theorem: if $\F$ is an acyclic family $F$ in a locally arc-wise connected topological space then $\homred_\ell(\M(\F))\cong\homred_\ell(\bigcup_\F)$ for any $\ell \ge 0$. We will use the following immediate consequence of Theorem~\ref{thm:multinerf}:

\begin{corollary}\label{cor:multinerf}
 If $\F$ is a finite acyclic family of open sets in  a locally arc-wise connected topological space $\Gamma$ then the
 Leray number of $\M(\F)$ is at most~$d_\Gamma$.
 \end{corollary}
  \begin{proof}
    Let $\G$ be a finite sub-family of~$\F$ and let $\ell\leq d_\Gamma$. Since $\M(\F[\G])=\M(\G)$ and $\G$ is also acyclic, Theorem~\ref{thm:multinerf} yields that $\homred_\ell(\M(\F[\G]))=\homred_\ell(\M(\G))\cong \homred_\ell(\bigcup\nolimits_\G)$ for any $\ell \ge 0$. Since $\bigcup_\G$ is open, $\homred_\ell(\bigcup\nolimits_\G) = 0$ for $\ell \ge d_\Gamma$. The statement follows.
 \end{proof}

\begin{remark}
We only use the assumption that $\Gamma$ be locally arc-wise connected to ensure that the connected components and the arc-wise connected components of any open subset of $\Gamma$ agree. It can be dispensed of by replacing the ordinary homology by the \v{C}ech homology (see~\cite{Br-sheaf-97}).
\end{remark}

\noindent
The rest of this section proves Theorem~\ref{thm:multinerf} via a reformulation of Leray's acyclic cover theorem. The reader unfamiliar with algebraic topology and willing to admit Theorem~\ref{thm:multinerf} can proceed to Section~\ref{sec:projection}.

\subsection{The chain complex of the multinerve}
\label{sec:cech}

To compute the homology of a multinerve, we first reformulate its associated chain complex (as given in Section~\ref{sec:posets}) in topological terms. If $X$ is a topological space we let $\pi_0(X)$ denote the set of arc-wise connected components of $X$. Since $\Gamma$ is locally arc-wise connected, for any open subset of $\Gamma$, the connected components agree with the arc-wise connected components. Hence, for any $n\in \mathbb{N}$, the set of (non-degenerate) $n$-simplices of $\M(\F)$ rewrites as $\M(\F)_n:=\coprod_{|A|=n+1}  \pi_0(\bigcap_{A})$. For every $X_i\in A$ and $i\in\{0,\dots, |A|-1\}$, the face map $d_{A,i}$ is the linear map $d_{A,i}:{H}_0(\bigcap_A)\to  {H}_0(\bigcap_{A\setminus X_i})$ induced by the inclusion.  (Here $H_0$ stands for the degree $0$ homology group.)

Now, the chain complex of the multinerve $(C_{n\geq 0}(\M(\F)),d)$ is the vector space over $\mathbb{Q}$ spanned by the set $\{(C,A)\in \M(\F),\, |A|=n+1 \}$; the differential maps (up to sign) a connected component $C$ of $\bigcap_A$ to the connected component $C'$ of $\bigcap_{A'}$ that contains $C$ for any $A' \subset A$ with $|A'|=|A|-1$. Using the geometric interpretation of the simplices of the multinerve, we can simplify its chain complex as follows.

\begin{lemma}\label{L:multinerveasCech}
  The chain complex $(C_{n\geq 0}(\M(\F)), d)$ is the chain complex
  $C_n(\M(\F))=\bigoplus_{|A|=n+1}H_0(\bigcap_A)$ whose differential is the
  linear map $d:C_n(\M(\F))\to C_{n-1}(\M(\F))$ given by
  $d=\sum_{i=0}^{n} (-1)^i d_{A,i}$.
\end{lemma}
\begin{proof}
  The lemma follows from the observation that the  degree 0
  homology ${H}_0(\bigcap_A)$ is isomorphic to $\bigoplus_{C\subseteq
    \bigcap_A}\mathbb{Q}$, where $C$ goes through all connected components
  of $\bigcap_A$.
\end{proof}

Given a (locally arc-wise connected) topological space $X$, the rule that
assigns to an open subset $U\subseteq X$ the set $\pi_0(U)$ of its (arc-wise)
connected component is a cosheaf on $X$. Taking $X=\bigcup_{\F}$, and
assuming that the elements of $\F$ are open sets in $X$, the family
$\mathcal{F}$ is an open cover of $X$. It follows from
Lemma~\ref{L:multinerveasCech} that the chain complex of $\M(\F)$ is
isomorphic to the \v{C}ech complex $\check{C}(\F,\pi_0)$ of the cosheaf
$U\mapsto \pi_0(U)$. 

\begin{remark}\label{rem:Cechmultinerve}
When the space is not locally arc-wise connected, Lemma~\ref{L:multinerveasCech} still applies if $H_0(\bigcap_A)$ is replaced by $\check{H}_0(\bigcap_A)$, the $\mathbb{Q}$-vector space generated by connected components of $\bigcap_A$. 
\end{remark}
\subsection{Proof of the homological multinerve theorems}

We write $\big(S_\bullet(X),d^S\big)$ for the singular chain complex of a
topological space $X$ that computes its homology. We also write
$C_\bullet(\M(\F))$ for the simplicial chain complex computing the
simplicial homology of the multinerve $\M(\F)$.
  
For any open subsets $U\subseteq V\subseteq X$, there is a natural chain
complex map $S^\bullet(U) \to S^\bullet(V)$ and thus the rule 
$U\mapsto S^\bullet(U)$ is a precosheaf on $X$ (but not a cosheaf in general). 
There is a standard way to replace this precosheaf by a cosheaf. Indeed, following~\cite[Section VI.12]{Br-sheaf-97}, there is a chain complex of cosheaves  $U\mapsto \mathfrak{S}_\bullet(U)$ (where $U$ is an open subset in $X$) which comes with canonical isomorphisms $H_n(U) \cong H_n(\mathfrak{S}_\bullet(U))$. We refer to~\cite[Section VI.12]{Br-sheaf-97} for the precise definition of $\mathfrak{S}_\bullet$, we only use the fact that it is a (differential) cosheaf inducing the above isomorphisms in homology (and a result from~\cite{Br-sheaf-97}). We write $d^{\mathfrak{S}}:\mathfrak{S}_\bullet(-) \to \mathfrak{S}_{\bullet-1}(-) $ for the differential on $\mathfrak{S}_\bullet(-)$.

\smallskip

We recall the notion of the \v{C}ech complex  of a (pre)cosheaf associated to a cover which is just the dual of the more classical notion of \v{C}ech complex of a (pre)sheaf (see~\cite{Br-sheaf-97, Go-taf-73, BoTu} for more details). Let $X$ be a topological space and $\mathcal{U}$ be a cover of $X$ (by open subsets). Also let $\mathfrak{A}$ be a precosheaf of abelian groups on $X$, that is, the data of an abelian group $\mathfrak{A}(U)$ for every open subset $U\subseteq X$ with corestriction (linear) maps $\rho_{U\subseteq V}: \mathfrak{A}(U) \to \mathfrak{A}(V)$ for any inclusion $U\hookrightarrow V$ of open subsets of $X$ which satisfies the coherence rule $\rho_{V\subseteq W} \circ \rho_{U\subseteq V}=\rho_{U\subseteq W}$ for any open sets $U\subseteq V\subseteq W$. The degree $n$ part of the \v{C}ech complex $\check{C}_n(\mathcal{U},\mathfrak{A})$ of the cover $\mathcal{U}$ with value in $\mathfrak{A}$ is, by definition,  $\bigoplus \mathfrak{A}\big(\bigcap_{I}\big)$ where the sum is over all subsets $I\subseteq \mathcal{U}$ such that $|I|=n+1$ and the intersection $\bigcap_I$ is non-empty. In other words, the sum is over all dimension $n$-simplices of the nerve of the cover $\mathcal{U}$. The differential $d$ is the sum $d=\sum_{i=0}^n (-1)^i d_{I,i}$ where $d_{I,i}:\mathfrak{A}\big(\bigcap_I\big) \to \mathfrak{A}\big(\bigcap_{I\setminus i}\big)$ is defined as in Lemma~\ref{L:multinerveasCech}, with $\mathfrak{A}$ instead of $H_0$.

\smallskip

Specializing to the case $X=\bigcup_\F$, we have a cover of $X$ given by the family $\F$. Thus we can now form the \v{C}ech complex $\check{C}(\F,\mathfrak{S}^\bullet(-))$ of the cosheaf of complexes $U\mapsto \mathfrak{S}^\bullet(U)$. Explicitly, $\check{C}(\F,\mathfrak{S}^\bullet(-))$ is the bicomplex $\check{C}_{p,q}(\F,\mathfrak{S}^\bullet(-))=\bigoplus_{|\G|=p+1} \mathfrak{S}^q(\bigcap_{\G})$ with (vertical) differential $d_v:\bigoplus_{|\G|=p+1} \mathfrak{S}^q(\bigcap_{\G}) \to \bigoplus_{|\G|=p+1} \mathfrak{S}^{q-1}(\bigcap_{\G})$ given by $(-1)^{p}d^{\mathfrak{S}}$ on each factor and with (horizontal) differential given by the usual \v{C}ech differential, that is, $d_h:\bigoplus_{|\G|=p+1} \mathfrak{S}^q(\bigcap_{\G}) \to \bigoplus_{|\G|=p} \mathfrak{S}^{q}(\bigcap_{\G})$ is the alternate sum $d_v=\sum_{i=0}^{|\G|} (-1)^i d_{\G,i}$ with the same notations as in Lemma~\ref{L:multinerveasCech}.

It is folklore that the homology of the (total complex associated to the)
bicomplex is the (singular) homology
$H_\bullet(\bigcup_{\F})=H_\bullet(\Gamma)$, (see~\cite{BoTu} for the cohomological version of it). More precisely, 
\begin{lemma}\label{lem:folklore}
There are natural isomorphisms 
\[H_n^{tot}(\check{C}_{\bullet,\bullet}(\F,\mathfrak{S}^\bullet(-))) \cong H_n\Big(\bigcup\nolimits_\F\Big)\]
where $H_n^{tot}(\check{C}_{\bullet,\bullet}(\F,\mathfrak{S}^\bullet(-)))$ is the homology of the (total complex associated to the) \v{C}ech bicomplex $\check{C}_{\bullet,\bullet}(\F,\mathfrak{S}^\bullet(-))) $. 
\end{lemma}
\begin{proof}
 Since $\check{C}_{\bullet,\bullet}(\F,\mathfrak{S}^\bullet(-))) $ is a bicomplex,
 by a standard argument recalled in Appendix~\ref{A:specseq}, the filtration by the columns of $\check{C}_{\bullet,\bullet}(\F,\mathfrak{S}^\bullet(-))) $ yields a spectral sequence $F^1_{p,q} \Rightarrow H_{p+q}^{tot}(\check{C}_{\bullet,\bullet}(\F,\mathfrak{S}^\bullet(-)))$. Since the horizontal differential is the \v{C}ech differential, the first page $F^1_{p,q}= \check{H}_p\big( \F,\mathfrak{S}_q(-)\big)$ is isomorphic to the \v{C}ech homology of the cosheaves $\mathfrak{S}_q(-)$ associated to the cover (of $\Gamma$) given by the family $\F$.  By Proposition VI.12.1 and Corollary VI.4.5 in~\cite{Br-sheaf-97}, these homology groups vanishes for $p>0$. Thus, $F^1_{p,q}=0$ if $q>0$ and $F^1_{p,0}\cong \mathfrak{S}_p(\bigcup_{\F})=\mathfrak{S}_p(\Gamma)$. Recall from Appendix~\ref{A:specseq} that the differential $d^1$ on the first page $F^1_{\bullet,\bullet}$ is given by the vertical differential $d_v=\pm d^{\mathfrak{S}}$.   Since, by definition, $H_n(\mathfrak{S}_\bullet(\Gamma), d^{\mathfrak{S}})\cong H_n(\Gamma)$, it follows that $F^2_{p,q}=0$ if $q>0$ and $F^2_{p,0}\cong H_p(\Gamma)$. Now, for degree reasons, all higher differentials $d^r: F^r_{\bullet, \bullet} \to F^{r}_{\bullet, \bullet}$ are zero. Thus $F^{\infty}_{p,q}\cong F^2_{p,q}$ and it follows that $ H_{n}^{tot}(\check{C}_{\bullet,\bullet}(\F,\mathfrak{S}^\bullet(-)))\cong F^2_{n,0} \cong H_n(\Gamma)$.
\end{proof}

By the standard argument recalled in Appendix~\ref{A:specseq} and Lemma~\ref{lem:folklore}, there is a converging spectral sequence (associated to the filtration by the rows of $\check{C}(\F,\mathfrak{S}_\bullet(-))$) $E^1_{p,q}\Rightarrow H_{p+q}(\bigcup_\F)$ such that $E^1_{p,q}=\bigoplus_{|\G|=p+1} H_q(\bigcap_{\G})$ and the differential $d^1:E^1_{p,q}\to E^1_{p-1,q}$ is (induced by) the horizontal differential $d_h$.  By Lemma~\ref{L:multinerveasCech}, there is an isomorphism $(E^1_{\bullet,0},d^1)\cong (C_\bullet(\M(\F)),d)$ of chain complexes and thus the bottom line of the page $E^2$ of the spectral sequence $E^2_{p,0}\cong H_p(\M(\F))$ is the homology of the multinerve of $\F$. The proof of Theorem~\ref{thm:multinerf} now follows from a simple analysis of the pages of this spectral sequences.

\begin{proof}[Proof of Theorem~\ref{thm:multinerf}] 
  By assumption, for any $q \ge \max(1,d-p-1)$ and any sub-family $\G
  \subseteq \F$ with $|\G|=p+1$, we have $H_q(\bigcap_\G)=0$ and thus
  $E^1_{p,q}=0$ for $q \ge \max(1,d-p-1)$. Since, for $r\ge 1$, the
  differential $d^r$ maps $E^r_{p,q}$ to $E^r_{p-r,q-1+r}$, by induction,
  we get that the restriction of $d^r$ to $E^r_{p,q}$ is null if $q\ge 1$
  and $p+q\ge d-1$.  Further $E^2_{p,0}\cong H_p(\M(\F))$ and, again for
  degree reasons, it follows that, for $r\geq 2$, $d^r:E^r_{p,0} \to
  E^r_{p-r,r-1}$ is null if $p\ge d$.

  Since $E^{r+1}_{\bullet,\bullet}$ is isomorphic to the homology $
  H_\bullet(E^r_{\bullet, \bullet}, d^r)$, it follows from the above
  analysis of the differentials $d^r$ that, for $p+q\ge d$ and $q\ge 1$,
  one has $E^2_{p,q}\cong 0$ and further that $E^2_{p,q}\cong
  E^3_{p,q}\cong \cdots \cong E^\infty_{p,q}$ for $p+q\ge d$. Now, we use
  that the spectral sequence converges to $H_{\ell}(\bigcup_\F)$. Hence,
  for any $\ell\geq d$, we find
$$H_{\ell}\Big(\bigcup_\F\Big) \cong \bigoplus_{p+q\ge \ell} E^\infty_{p,q} \cong  \bigoplus_{p+q\ge \ell} E^2_{p,q} \cong E^2_{\ell,0}\cong H_\ell(\M(\F)).$$
\end{proof}

\section{Projection of a simplicial poset}
\label{sec:projection}

\begin{figure}[tb]
  \centerline{\includegraphics[width=14cm,
    keepaspectratio]{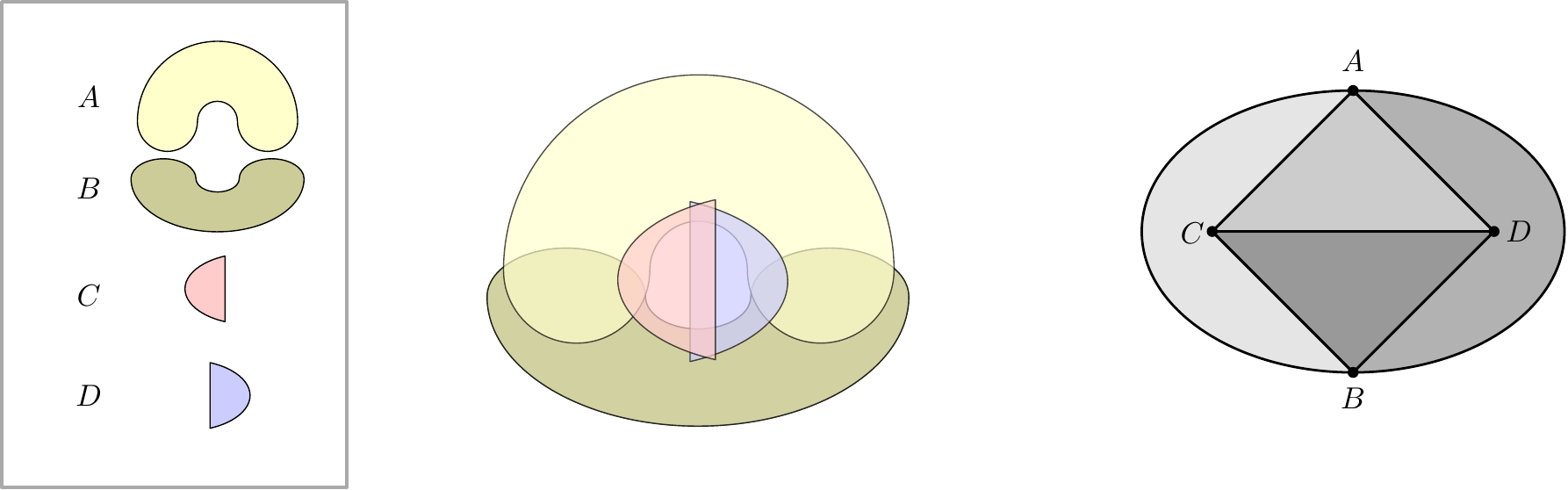}}
  \caption{An acyclic family (middle, with its elements
    represented individually on the left for clarity) whose multinerve
    (right) is contractible and whose nerve ($2^{\{A,B,C,D\}} \setminus
    \{A,B,C,D\}$) has nonzero homology in dimension $2$.}
  \label{fig:increasing-homology}
\end{figure}

With Corollary~\ref{cor:multinerf}, we see that the conditions of Theorem~\ref{thm:mainresult} ensure that the Leray number of the multinerve of~$\F$ is at most~$d_\Gamma$. The statement of Theorem~\ref{thm:mainresult} does, however, pertain to the nerve of $\F$, which is only a projection of the multinerve, and projecting a simplicial poset may increase the homology, as measured by the Leray number (see Figure~\ref{fig:increasing-homology} for an example). In this section, we show that under certain conditions, this accession can be controlled.

\subsection{Statement of the projection theorem}
\label{S:statementprojection}
Let $X$ be a simplicial poset with vertex set~$V$.  Recall that the Leray
number of~$X$ is defined as:
\[L(X)=\min\Setbar{\ell}{\forall j\ge\ell, \forall S\subseteq V,
  \homred_j(X[S])=0}.\]

We define $J(X)$ to be the smallest integer~$\ell$ such that for every
$j\geq\ell$, every $S\subseteq V$, and every simplex~$\sigma$ of~$X[S]$, we
have $\homred_j(\dD_{X[S]}(\sigma))=0$.

\begin{lemma}\label{lem:jl}
  $L(X)\le J(X)$.
\end{lemma}
\begin{proof}
  Let $S\subseteq V$ and let $0$ be the least element of~$X$.  By
  definition, $\dD_{X[S]}(0)$ is the barycentric subdivision
  of~$X[S]$.  Thus, by definition of~$J(X)$, for every $j\geq J(X)$, we
  have $\homred_j(X[S])=0$.  Thus $L(X)\le J(X)$.
\end{proof}

The purpose of this section is to prove the following projection theorem.
\begin{theorem}\label{thm:new-proj}
  Let $r\geq 1$. Let $\pi$ be a monotone map from a simplicial
  poset~$X$ to a simplicial complex~$Y$. Assume that for every
  simplex~$\tau$ in~$Y$, $\pi\inv(\tau)$ contains between $1$ and~$r$
  simplices of~$X$, all of the same dimension as~$\tau$. Then
  $L(Y)\leq rJ(X)+r-1$.
\end{theorem}

We note that, since $Y$ is a simplicial complex, $L(Y)=J(Y)$: this follows
from \cite[Proposition~3.1]{km-ilcrm-06} and from the isomorphism between
$\dD_{Y[S]}(\sigma)$ and the barycentric subdivision of the link of~$\sigma$
in~$Y[S]$.  So the conclusion of the theorem can be rewritten as $J(Y)\leq
rJ(X)+r-1$; however, we will not use this result.  On the other hand, we
cannot decide if $L(X)=J(X)$ holds, because $X$ is merely a simplicial poset.
This is the reason for our introduction of the quantity~$J$.

The special case of Theorem~\ref{thm:new-proj} when $X$ is a simplicial
complex was proven by Kalai and Meshulam~\cite[Theorem 1.3]{km-lnpthtt-08}
in a slightly different terminology. We note that already in this context
the bound on $L(Y)$ is tight (see the remark after Theorem~1.3
of~\cite{km-lnpthtt-08}). 

In the remaining part of this section, we prove Theorem~\ref{thm:new-proj}.  Specifically, we describe how the proof of Kalai and Meshulam~\cite[Theorem 1.3]{km-lnpthtt-08}, once it is reformulated in our terminology, extends, mutantis mutandis, to the case of simplicial posets.  The reader not interested in the proof of Theorem~\ref{thm:new-proj} can safely proceed to Section~\ref{sec:proof}.

\subsection{Structure of the proof}

The proof of the projection theorem of Kalai and Meshulam \cite[Theorem~1.3]{km-lnpthtt-08} uses properties of the \emph{multiple $k$-point space} (defined below) in two independent steps, each using a different spectral sequence. The first step relates the homology of~$Y$ to that of the multiple $k$-point space.  The second, more combinatorial step, aims at controlling the topology of the multiple $k$-point space.

For the proof of their projection theorem, Kalai and Meshulam assume that
$X$ is a subset of the join of disjoint $0$-complexes $V_1\ast\ldots\ast
V_m$, where $\pi$ maps each vertex of~$V_i$ to the $i$th vertex
of~$Y$. Instead, we assume that $\pi:X\to Y$ is dimension-preserving. This
assumption is equivalent in the context of simplicial complexes (as can be
seen by taking $V_i=\pi\inv(i)$ for each vertex~$i$) and remains meaningful
for simplicial posets.

\subsection{The image computing spectral sequence}

The first spectral sequence considered~\cite[Theorem~2.1]{km-lnpthtt-08} is
due to Goryunov-Mond~\cite{gm-vcsm-93} and uses only topological properties
of the geometric realization and the fact that we are considering homology
with coefficient in the field $\Q$ of rational numbers. It thus extends
verbatim to the setting of simplicial posets.

Specifically, for $k\geq 1$, the \emph{multiple $k$-point space} $M_k$
of $X$ is $$M_k=\Set{(x_1,\dots,x_k)\in |X|^k \, \text{s.t.} \,
\pi(x_1)=\cdots=\pi(x_k)}.$$ Note that there is a natural action of
$S_k$, the symmetric group on $k$ letters, on $M_k$ by permutation,
and thus on the homology $H_\bullet(M_k)$ as well. We denote
$$\mathop{\text{Alt}}H_n(M_k)=\{v\in H_n(M_k),\; \sigma \cdot v =
\text{sgn}(\sigma) v \text{ for all }\sigma \in S_k\}.$$ The
(geometric realization of the) simplicial map $\pi$ being finite with
the sets $\pi^{-1}(y)$ (for any $y\in Y$) being of cardinality at most
$r$, we have the following result, which is the same as Theorem 2.1
in~\cite{km-lnpthtt-08}.

\begin{theorem}[Goryunov-Mond] \label{thm:GM}
  There is a homology spectral sequence $E^r_{p,q}$ converging to
  $H_\bullet(Y)$ such that $$E^1_{p,q}=\left\{
    \begin{array}{ll}\mathop{\text{Alt}}H_q(M_{p+1}) & \text{ for }
      0\leq p\leq r-1,\, 0\leq q \\ 0 & \text{ otherwise. }\end{array}
  \right .$$
\end{theorem}

\subsection{Homology of multiple point sets}

We now argue that $H_q(M_{p+1})=0$ for $q$ large enough. Let
$X_1,\dots, X_k $ be induced simplicial sub-posets of $X$.  Define
$$ M(X_1,\ldots,X_k)=\big\{(x_1,\ldots,x_k)\in |X_1|\times \dots\times |X_k|,\quad
\pi(x_1)=\cdots=\pi(x_k)\big\}.$$
Note that $M(X_1,\ldots,X_k)=M_k$. We are actually mainly interested
in the case $X_1=\dots=X_k=X$ but it is more convenient to have
different indices for bookkeeping issues in the proof. In our setting,
the analogue of Proposition~3.1 in~\cite{km-lnpthtt-08} is the following.

\begin{lemma}\label{lem:new-pro3.1}
  $\homred_j(M(X_1,\ldots,X_k))=0$ for $j\geq\sum_{i=1}^kJ(X_i)$.
\end{lemma} 
\begin{proof}
  Given $\sigma\in X$, we define $\tilde\sigma$ as the set of vertices
  of~$X$ in $\pi\inv(\pi(\sigma))$. We thus have that
  $M(X_1,\sigma_2,\ldots,\sigma_k)$ is homeomorphic to
  \[\big\{x_1\in |X_1|,\quad \forall i=2,\ldots,k,\,\exists x_i\in|\sigma_i|,\, \pi(x_i)=\pi(x_1)\big\},\]
  since the $x_2,\ldots,x_k$ are uniquely determined because $\pi$ is
  dimension-preserving. We thus have the following identification:
  \[ M(X_1,\sigma_2,\ldots,\sigma_k)\cong \left|X_1\left[\bigcap_{i=2}^k
      \tilde{\sigma}_i\right]\right|,\] which extends
  \cite[Equation~(3.1)]{km-lnpthtt-08}. Let $n=\sum_{j=2}^k
  \mathop{\text{dim}}(X_j)$; define the sets
  $$\mathcal{S}'_p=\left\{ (\sigma_2,\ldots, \sigma_k)\in X_2\times
    \cdots\times X_k, \, \sum_{j=1}^k
    \mathop{\text{dim}}(\sigma_j)\geq n-p\right\}$$ and
  $\mathcal{S}_p=\mathcal{S}'_p-\mathcal{S}'_{p-1}$ for $0\leq p\leq
  n$. Furthermore, for $(\sigma_2,\ldots,\sigma_k)\in\mathcal{S}'_p$,
  define \[A_{(\sigma_2,\ldots,\sigma_k)}=M(X_1,\sigma_2,\ldots,\sigma_k)
  \times D_{X_2}(\sigma_2)\times\ldots\times D_{X_k}(\sigma_k).\]

Now, consider the spaces %
$$K_p=\bigcup_{(\sigma_2,\ldots,\sigma_k)\in\mathcal{S}'_p} A_{(\sigma_2,\ldots,\sigma_k)} \quad\subseteq\, M(X_1,\dots,X_k)\times \sd(X_2)\times \cdots \times \sd(X_k).$$ %
Since the $D_{X_i}(\sigma_i)$ are contractible, it follows that the projection on the first coordinate $K_n\to M(X_1,\ldots,X_k)$ is a homotopy equivalence and the homology spectral sequence associated to the filtration $\emptyset\subset K_0\subset \cdots\subset K_n$ converges to $H_\bullet(M(X_1,\ldots,X_k))$ and is analogous to the one given in~\cite[Proposition~3.2]{km-lnpthtt-08}. The first page of this spectral sequence writes $E^0_{p,q}= H_{p+q}(K_p, K_{p-1})$. The arguments used by~\cite[Proposition~3.2]{km-lnpthtt-08} for the identification of the second page, that is the $E^1_{p,q}$-terms, are based on properties of the homology of pairs such as excision and K\"unneth formula. Since the barycentric subdivision of a simplicial poset is itself a simplicial complex, these arguments extend to our setting and we get that
  $$E^1_{p,q}\cong \bigoplus_{\substack{(\sigma_2,\ldots,\sigma_k)\\\in\ \SSS_p}}
  \bigoplus_{\substack{i_1,\ldots,i_k\geq0\\i_1+\ldots+i_k=p+q}}
  H_{i_1}\left(X_1\left[\bigcap_{i=2}^k\tilde\sigma_i\right]\right)
  \otimes\bigotimes_{j=2}^k H_{i_j}\left(D_{X_j}(\sigma_j),\dot
    D_{X_j}(\sigma_j)\right).$$

  In the simplicial complex setting, Kalai and Meshulam then use the isomorphism between $\dot D_{X_j}(\sigma_j)$ and the barycentric subdivision of the link of $\sigma_j$ in~$X_j$ together with a characterization of Leray numbers in terms of reduced homology of all links in the simplicial complex~\cite[Proposition~3.1]{km-ilcrm-06}.  The introduction of~$J(X)$ in our setting will circumvent the fact that the notion of link does \emph{not} extend well to simplicial posets.  Since $D_{X_j}(\sigma_j)$ is contractible, we still have $H_{i_j}(D_{X_j}(\sigma_j),\dot D_{X_j}(\sigma_j)) \cong \homred_{i_j-1}(\dot D_{X_j}(\sigma_j))$.  This yields the identification
\begin{equation}
  E^1_{p,q}\cong\bigoplus_{\substack{(\sigma_2,\ldots,\sigma_k)\\\in\ \SSS_p}}
  \bigoplus_{\substack{i_1,\ldots,i_k\geq0\\i_1+\ldots+i_k=p+q}}
  H_{i_1}\left(X_1\left[\bigcap_{i=2}^k\tilde\sigma_i\right]\right)
  \otimes\bigotimes_{j=2}^k\homred_{i_j-1}\left(\dot D_{X_j}(\sigma_j)\right).
\end{equation}

We now have all the ingredients to finish the proof of the lemma.  First note that for a simplicial complex $L(Z)=0$ implies that~$Z$ is a simplex; this is still true if $Z$ is a simplicial poset. Let $m=\sum_{j=1}^kJ(X_j)$.  If $m=0$, then, by Lemma~\ref{lem:jl}, $M(X_1,\ldots,X_k)$ is isomorphic to a simplex and has no reduced homology in all non-negative dimensions. We can thus assume $m>0$. Since we have a homology spectral sequence $E^1_{p,q}$ converging to $H_\bullet(M(X_1,\dots, X_k))$, it suffices to prove that $E^1_{p,q}=0$ when $p+q=i_1+\cdots+i_k\geq m$. If $i_1\geq J(X_1)$, we have $i_1\geq L(X_1)$ by Lemma~\ref{lem:jl} and therefore $H_{i_1}\left(X_1\left[\bigcap_{i=2}^k\tilde\sigma_i\right]\right)=0$.  Furthermore, if $i_j-1\geq J(X_j)$, then by definition we have $\homred_{i_j-1}(\dot D_{X_j}(\sigma))=0$. Thus, if $p+q\geq M=\sum_{j=1}^kJ(X_j)$, at least one of the tensors in
$$H_{i_1}\left(X_1\left[\bigcap_{i=2}^k\tilde\sigma_i\right]\right)
\otimes\bigotimes_{j=2}^k\homred_{i_j-1}\left(\dot
  D_{X_j}(\sigma_j)\right)$$ is null and it follows that $E^1_{p,q}=0$.
This concludes the proof.
\end{proof}

\subsection{End of the proof of Theorem~\ref{thm:new-proj}}

\begin{lemma}\label{lem:inter-new-proj}
  $\homred_\ell(Y)=0$ if $\ell\geq rJ(X)+r-1$.
\end{lemma}
\begin{proof}
  If $J(X)=0$, we are left to the case where $X$ is a simplex and there is
  nothing to prove. Thus we may assume $J(X)>0$. By Theorem~\ref{thm:GM},
  it suffices to prove that $E^1_{p,q}\cong
  \mathop{\text{Alt}}H_q(M_{p+1})=0$ if $p+q\geq rJ(X)+r-1$ with $p\leq
  r-1$ and $q\geq 0$. Since $M_{p+1}\cong M(X_1,\dots, X_{p+1})$ for
  $X_1=\cdots=X_{p+1}=X$, by Lemma~\ref{lem:new-pro3.1}, we have that
  $H_q(M_{p+1})=0$ for $q\geq (p+1)J(X)$. Now the conditions $p+q\geq
  rJ(X)+r-1$ and $p\leq r-1$ imply $q\geq rJ(X)\geq (p+1)J(X)$ and thus
  that $H_q(M_{p+1})=0$. There is nothing left to prove.
\end{proof}

We conclude:
\begin{proof}[Proof of Theorem~\ref{thm:new-proj}]
  Let $S$ be a subset of vertices of~$Y$ and let $R=\pi\inv(S)$.  We apply
  Lemma~\ref{lem:inter-new-proj} with $X[R]$ and~$Y[S]$, which also satisfy
  the hypotheses of the theorem.  We obtain $\homred_\ell(Y[S])=0$ if
  $\ell\geq rJ(X[R])+r-1$.  By definition of~$J$, we have $J(X[R])\leq
  J(X)$; so we have $L(Y)\leq rJ(X)+r-1$, as desired.
\end{proof}

\section{Topological Helly-type theorems for acyclic families}
\label{sec:proof}

We now put everything together to prove our main results, Theorems~\ref{thm:mainresult} and~\ref{thm:non-uniform}, and conclude this section by showing that the openness condition can be replaced, in a slightly less general context, by a compactness condition. 

\subsection{Proof of Theorem~\ref{thm:mainresult}}

Our first step towards a proof of Theorem~\ref{thm:mainresult} is to bound from above the $J$ index of the multinerve of an acyclic family.  For future reference, we actually allow the family to have some slack.

\begin{lemma}\label{lem:bound-j}
Let $\Gamma$ be a locally arc-wise connected topological space. If $\F$ is a finite family of open subsets of $\Gamma$ that is acyclic with slack $s$, then $J(\M(\F))\le \max(d_\Gamma,s)$.
\end{lemma}
\begin{proof}
  Let $\G\subseteq\F$ be a sub-family of~$\F$, and let $\sigma$ be a simplex
  of $\M(\F)[\G]=\M(\G)$.  We need to prove that $\dD_{\M(\G)}(\sigma)$ has
  trivial reduced homology in dimension $\max(d_\Gamma,s)$ and higher.

  Given $\sigma =(C,A) \in \M(\G)$, we define $\G_{\sigma}$ as the
  non-empty traces of the elements of~$\G\setminus A$ on~$C$:
  \[ \G_{\sigma} = \{ U \cap C \mid U \in \G\setminus A, U\cap C \neq
  \emptyset\}.\] %
  (Note that $\G_{\sigma}$ is a multiset, as a given element may appear
  more than once.) The map
  \[\left\{ \map{\M(\G_\sigma)}{[\sigma,\cdot]}{(C',A')}{(C'\cap C,A\cup
      A')}\right.\] %
  is an isomorphism of posets. In particular, $[\sigma,\cdot]$ is a
  simplicial poset.  Both posets have a least element, and removing them
  yields that $\M(\G_{\sigma}) \setminus \{(\bigcup_\G, \emptyset)\}$ and
  $(\sigma, \cdot]$ are isomorphic posets. Taking their order complexes, we
  get that $\dD_{\M(\G)}(\sigma)$ and $\sd(\M(\G_{\sigma}))$ are isomorphic
  simplicial complexes.

  Therefore, $\dD_{\M(\G)}(\sigma)$ has the same homology as $\M(\G_{\sigma})$.  Since $\F$ is acyclic (with slack $s$), the family $\G_{\sigma}$ is acyclic (with slack $s$) as well.  Theorem~\ref{thm:multinerf} now ensures that (in dimension $j \ge s$) the homology of $\M(\G_{\sigma})$ is the same as the homology of the union of the elements in~$\G_\sigma$. Since $\bigcup_{\G_\sigma}$ is an open subset of~$\Gamma$, it has homology zero in dimension~$d_\Gamma$ and higher.  This concludes the proof. 
\end{proof}

Our first Helly-type theorem now follows easily through our projection theorem.

\begin{proof}[Proof of Theorem~\ref{thm:mainresult}]
 Let $\F$ be a finite acyclic family of open subsets of a locally arc-wise connected topological space $\Gamma$, and assume that any sub-family of $\F$ intersects in at most $r$ connected components. Let $\N(\F)$ and $\M(\F)$ denote, respectively, the nerve and the multinerve of $\F$. We consider the projection
  \[ \pi: \left\{ \begin{array}{rcl} \M(\F) & \rightarrow & \N(\F) \\
      (C,A) & \mapsto & A\end{array}\right.\]
  (already used in Section~\ref{sec:def-multinerf}). Each simplex in the pre-image $\pi\inv(\sigma)$ of a simplex $\sigma \in \N(\F)$ is of the form $(C,\sigma)$ where $C$ is a connected component of $\bigcap_\sigma$. The projection $\pi$ is therefore at most $r$-to-one and we can apply Theorem~\ref{thm:new-proj} with $X=\M(\F)$ and $Y=\N(\F)$.  We obtain that $L(\N(\F)) \le rJ(\M(\F)) +r-1$. With Lemma~\ref{lem:bound-j}, this becomes $L(\N(\F)) \le r(d_\Gamma+1)-1$. Since the Helly number of $\F$ is at most $L(\N(\F))+1$, as argued in Section~\ref{sec:posets}, this concludes the proof.
\end{proof}

\subsection{Proof of Theorem~\ref{thm:non-uniform}}

Before we move on to proving our more general Helly-type theorem we need another (simple) projection theorem for the $J$ index.

\begin{lemma}\label{lem:stabletrivialfiber}
Let $X$ and $Y$ be two simplicial posets and $k\ge0$. If there exists a monotone, dimension-preserving map $f:X \to Y$ whose restriction to the  simplices of $X$ of dimension at least $k$ is a bijection onto the simplices of $Y$ of dimension at least $k$, then $J(Y)\;\leq \;\max\big(J(X), k+1\big)$.
\end{lemma}
\begin{proof}
Since $f$ is monotone, it induces a map $\tilde{f}: \sd(X) \to \sd(Y)$. We note that $\tilde{f}$ is clearly monotone and is also dimension-preserving, since $f$ is dimension-preserving.

Any $n$-simplex of $\sd(Y)$ is a chain of $n+1$ elements of $Y$ of increasing dimensions whose maximal element has therefore dimension at least $n$. For $n \ge k$, any $n$-simplex $\tau \in Y$ has a unique pre-image $\sigma \in X$ under $f$. Thus, for any chain $\upsilon$ in $Y$ with maximal element $\tau$, if $\upsilon$ has a pre-image under $f$ then the maximal element of that pre-image is $\sigma$. Since, by Lemma~\ref{lem:dimpres}, $f$ is a bijection from $[0,\sigma]$ onto $[f(0),\tau]$, it follows that any chain in $Y$ whose maximal element has dimension at least $k$ has one, and only one, pre-image under $f$. In particular, for any $n \ge k$ we have that $\tilde{f}$ induces a bijection from the $n$-simplices of $\sd(X)$ onto the $n$-simplices of $\sd(Y)$. 

Now let $V$ be the set of vertices of $Y$, let $S$ be a subset of $V$, and let $\tau$ be a simplex in $Y[S]$. Let $R = \bigcup_{f\inv(S)}$ and let $\{\sigma_1, \ldots, \sigma_p\}$ be the pre-images of $\tau$ through $f$. For every $n\ge k$, the map $f$ induces a bijection between the union of the $n$-simplices of~$X[R]$ containing  one of the~$\sigma_i$, and the set of $n$-simplices of~$Y[S]$ containing~$\tau$.  (It is actually a disjoint union.)  Thus, the same argument as above implies that, for every~$n\ge k$, $\tilde f$ induces a bijection between the $n$-simplices of $\bigcup_i\dD_{X[R]}(\sigma_i)$ and those of $\dD_{Y[S]}(\tau)$.

Furthermore, by Lemma~\ref{lem:dimpres}, both $f$ and $\tilde f$ (trivially extended by linearity) commute with the boundary operator.  The two previous statements imply that for every $n\geq k+1$, $\tilde f$ induces an isomorphism between $H_n(\bigcup_i \dD_{X[R]}(\sigma_i))$ and $H_n(\dD_{Y[S]}(\tau))$.  Since the union $\bigcup_i \dD_{X[R]}(\sigma_i)$ is actually disjoint, the homology group $H_n(\bigcup_i \dD_{X[R]}(\sigma_i))$ is $\bigoplus_i(H_n(\dD_{X[R]}(\sigma_i)))$.  By definition, all the summands vanish for $n\geq J(X)$.  Therefore, $H_n\big(\dD_{Y[S]}(\tau)\big)$ vanishes for any $S \subseteq V$, any $\tau \in Y[S]$ and any $n \ge \max(J(X),k+1)$.
\end{proof}
We can now prove our more general Helly-type theorem.

\begin{proof}[Proof of Theorem~\ref{thm:non-uniform}]
Let $\Gamma$ be a locally arc-wise connected topological space and let  $\F$ be a family of open subsets of $\Gamma$ that is acyclic with slack $s$ and such that the intersection of any sub-family of $\F$ of size at least $t$ has at most $r$ connected components.  Let $\N(\F)$ and $\M(\F)$ denote, respectively, the nerve and the multinerve of $\F$. We can construct a simplicial poset $\M_{red}(\F)$ by identifying together two simplices of~$\M(\F)$ if and only if they are of the form $(C,A)$ and~$(C',A')$ with $A=A'$ and $|A|\leq t-1$.
 In other words, 
\begin{multline*}
  \M_{red}(\F) =\Setbar{A}{ A\subseteq\F \hbox{ has cardinality at most } t-1 \hbox{ and } \bigcap\nolimits_A \neq \emptyset}\\
\cup \Setbar{(C,A)}{A \subseteq \F \hbox{ with cardinality at least $t$ and } C \hbox{ is a connected component of } \bigcap\nolimits_A}. \end{multline*}
 We thus have a map $f:\M(\F) \to \M_{red}(\F)$ given by $f(C,A)=(C,A)$ if $A$ has cardinality at least $t$ and $f(C,A)=A$ otherwise. 
The poset structure of $\M_{red}(\F)$ is similar to the one of the multinerve in Section~\ref{sec:def-multinerf} and the proof of Lemma~\ref{L:multinervesimplicial} applies mutatis mutandis to prove that $\M_{red}(\F)$ is a simplicial poset. We note that $f$ is monotone and dimension-preserving. Moreover, for any $n \ge t-1$, $f$ is a bijection from the $n$-simplices of $\M(\F)$ onto the $n$-simplices of $\M_{red}(\F)$. We can thus apply Lemma~\ref{lem:stabletrivialfiber} with $X=\M(\F)$, $Y=\M_{red}(\F)$ and $k=t-1$, and obtain that $J(\M_{red}(\F)) \le \max (J(\M(\F)), t)$. Since $J(\M(\F)) \le \max(d_\Gamma,s)$ by Lemma~\ref{lem:bound-j}, it follows that $J(\M_{red}(\F)) \le \max (d_\Gamma,s,t)$. 

Now, consider the projection $\pi: \M_{red}(\F) \to \N(\F)$ that is the identity on simplices of dimension at most $t-2$ and such that for any simplex $(C,A) \in \M_{red}(\F)$ of dimension at least $t-1$, $\pi(C,A) = A$. By construction, $\pi$ is at most $r$-to-one, so we can apply Theorem~\ref{thm:new-proj} with $X=\M_{red}(\F)$ and $Y=\N(\F)$ to obtain that $L(\N(\F)) \le rJ(\M_{red}(\F)) +r-1$. Since $J(\M_{red}(\F)) \le \max (d_\Gamma,s,t)$, we get that $L(\N(\F))$ is at most  $r(\max (d_\Gamma,s,t)+1)-1$ and the statement follows.
\end{proof}

\subsection{Extension to compact sets}

We finally argue that the openness assumption can be replaced by a
compactness assumption under a mild additional condition on the sets.
Here, by a \emph{triangulation} of a topological space~$\Gamma$, we
  mean a triangulation with finitely many simplices.

\begin{lemma}\label{lem:ext-closed-open}
  Let $\F$ be a finite family of subcomplexes of a triangulation~$T$ of an
  arbitrary topological space~$\Gamma$.  Then there exists a family
  $(O(F))_{F\in\F}$ of open sets in~$\Gamma$ such that, for every
  $\G\subseteq\F$, the set $\bigcap_{G\in\G}O(G)$ deformation retracts to
  $\bigcap_{G\in\G}G$.
\end{lemma}
\begin{proof}
  For an arbitrary subcomplex~$K$ of~$T$, let $O(K)$ be the union of the
  \emph{open} simplices of~$\sd(T)$ whose closure meets~$K$.  (By a slight
  abuse of notation, we also denote by $O(K)$ the \emph{set} of these
  simplices.)  It is a standard fact~\cite[Lemma~70.1]{m-eat-84} that
  $O(K)$ deformation retracts to~$K$: indeed, every simplex of $O(K)$ has a
  unique maximal face entirely contained in~$K$; the retraction collapses
  each such simplex of~$O(K)$ towards this maximal face.

  Let $\sigma$ be a simplex in~$\sd(T)$.  It is thus a chain of simplices
  in~$T$; let $\min(\sigma)$ be the simplex of~$T$ of smallest dimension in
  this chain.  With this notation, $\sigma\in O(K)$ if and only if
  $\min(\sigma)\in K$ (since $K$ is a subcomplex).  In other words,
  \[O(K)=\{\sigma\in\sd(T)\,\mid\,\min(\sigma)\in K\}.\] %
  This immediately implies that $O(K)$ is an open set and that, for every
  sub-family~$\G$ of~$\F$, we have $\bigcap_{G\in\G}O(G)=O(\bigcap_\G)$;
  this latter set retracts to~$\bigcap_\G$.
\end{proof}

In particular, the condition of being acyclic (with slack $s$) extends from a family $\F$ to the family $O(\F)$.  Theorems~\ref{thm:mainresult} and~\ref{thm:non-uniform} therefore extend immediately to subcomplexes of triangulations.  We only state the more general version:

\begin{corollary}\label{coro:compact}
Let $\F$ be a finite family of subcomplexes of a given triangulation of a locally arc-wise connected topological space $\Gamma$.  If (i) $\F$ is acyclic with slack $s$ and (ii) any sub-family of $\F$ of cardinality at least $t$ intersects in at most $r$ connected components, then the Helly number of~$\F$ is at most $r(\max(d_\Gamma,s,t)+1)$.
\end{corollary}

\section{Transversal Helly numbers}
\label{sec:transversal}

Let $\h=\{A_1, \ldots, A_n\}$ be a family of pairwise disjoint convex sets in $\R^d$ and let $T_k(\h)$  denote the set of $k$-dimensional affine subspaces intersecting every member in $\h$. Vincensini~\cite{v-fcvlee-35} conjectured that the Helly number of $\{T_k(A_1), \ldots, T_k(A_n)\}$, the \emph{$k$-th transversal Helly number} $\tau_k$ of $\{A_1, \ldots, A_n\}$, can be bounded as a function of $d$ and $k$, generalizing Helly's theorem that corresponds to the case $k=0$. Vincensini's conjecture is false in such generality but holds in special cases, when the geometry of the $A_i$ is adequately constrained. Understanding which geometric conditions allow for bounded transversal Helly numbers has been one of the focus of geometric transversal theory~\cite{dgk-htr-63,e-hrctt-93,h-rpltfto-08,w-httgt-04}. In this section we show that Theorem~\ref{thm:non-uniform} can be used to bound, in a single stroke, three transversal Helly numbers $\tau_1$ previously bounded via ad hoc methods. The parameters used in the applications of Theorem~\ref{thm:non-uniform} are summarized in Table~\ref{tab:parameters}.

For future reference, the following standard lemma bounds the value of~$d_\Gamma$ for some manifolds~$\Gamma$. The proof can be found in various textbooks, e.g. Greenberg~\cite[p.~121]{g-lat-67}.
\begin{lemma}\label{lem:homdim}
  Let $\Gamma$ be a (paracompact) manifold of dimension~$d$.  Then $d_\Gamma\leq d+1$.  Furthermore, if $\Gamma$ is non-compact or non-orientable, then $d_\Gamma\leq d$.
\end{lemma}

\subsection{General remarks}\label{ss:transgen}

Like most work in geometric transversal theory, we focus on the case $k=1$, when the subspaces are lines. We therefore give bounds on certain \emph{first transversal Helly numbers}. A line intersecting every member in $\h$ is called a \emph{line transversal} to $\h$. We let $T(\h)=T_1(\h)$ denote the set of line transversals to~$\h$. All lines are \emph{non-oriented}.

The space of lines in $\R^d$ can be considered as a subspace of the space of lines in $\R\p^d$, which is the Grassmannian $\RG_{2,d+1}$ of all $2$-planes through the origin in $\R^{d+1}$; $\RG_{2,d+1}$ is a manifold of dimension $2d-2$ and can be seen as an algebraic sub-variety of some $\R\p^m$ via Grassmann coordinates (also known as Pl\"ucker coordinates for $d=3$). We note that $d_{\RG_{2,d+1}}\le2d-1$ by Lemma~\ref{lem:homdim}.  However, in the applications below, we consider the set~$\Gamma$ of lines in~$\R^d$, which is a non-compact submanifold of dimension~$2d-2$ of~$\RG_{2,d+1}$.  It follows that $d_\Gamma\le2d-2$, again by Lemma~\ref{lem:homdim}.

Let $p:\RG_{2,d+1} \to \R\p^{d-1}$ be the map associating each line to its direction. We let $\K(\h) = p(T(\h))$ denote the directions of line transversals to $\h$. As the next lemma shows, the homology of $T(\h)$ can be studied through its projection by $p$.

\begin{lemma}\label{lem:directions}
If $\h$ is a finite family of compact convex sets in $\R^d$, then $p|_{T(\h)}$ induces an isomorphism in homology.  In other words, $p$ induces a bijection between the connected components of $T(\h)$ and the connected components of $\K(\h)$, and each connected component of $T(\h)$ has the same homology as its projection.
\end{lemma}
\begin{proof}
  For any direction $\vec{u} \in \K(\h)$ the fiber $p\inv(\vec{u})$ is contractible, as it is homeomorphic to the intersection of the projections of the members of $\h$ on a hyperplane orthogonal to $\vec{u}$. 
Furthermore, since $T(\h)$ is compact, the restriction $p_{|T(\h)}$ is a closed map.  Thus Lemma~\ref{lem:proj-trivial}(i) (in Appendix~\ref{A:contractible}) directly implies the result.
\end{proof}

The number of connected components of $T(\h)$ can be bounded under certain conditions on the geometry of the objects in $\h$. A line transversal to a family of disjoint convex sets induces two orderings of the family, one for each orientation of the line; this pair of orderings is called the \emph{geometric permutation} of the family induced by the line. A simple continuity argument shows that all lines in a connected component of $T(\h)$ induce the same geometric permutation of $\h$. Under certain conditions, this implication becomes an equivalence, and the connected components of $T(\h)$ are in one-to-one correspondence with the geometric permutations of $\h$. Various geometric and combinatorial arguments can then be used to bound from above the number of distinct geometric permutations that may exist for one and the same family $\h$.

An \emph{open thickening} of a subset $H$ of $\R^d$ is a family $(H^\varepsilon)^{\varepsilon>0}$ such that (i) any $H^\varepsilon$ is an open set, (ii) if $\varepsilon<\varepsilon'$, then $H^\varepsilon\subseteq H^{\varepsilon'}$, and (iii) $\bigcap_{\varepsilon>0}H^\varepsilon=H$.  For a family~$\G$ of subsets of~$\R^d$, we let $\G^\varepsilon=\Set{H^\varepsilon\mid H\in\G}$. In the three applications below, we consider transversals to \emph{compact} sets.  Since any compact set admits an open thickening, the following lemma will allow us to consider the same problem with \emph{open} sets.

\begin{lemma}\label{lem:transcompact}
  Let $\h$ be a finite family of compact convex sets in~$\R^d$ and
  $\h^\varepsilon$ be an open thickening of $\h$.  There exists
  $\varepsilon>0$ such that for every $\G\subseteq\h$, the family~$\G$
  has a common transversal if and only if the family $\G^\varepsilon$
  has a common transversal.
\end{lemma}
\begin{proof}
  Let $\G\subseteq\h$.  To prove the lemma, it suffices to prove that, if $\G$ has no transversal, then, for $\varepsilon>0$ small enough, $\G^\varepsilon$ has no transversal.  We prove the contrapositive statement: assume that $\G^\varepsilon$ has a transversal for every $\varepsilon>0$; we will prove that $\G$ has a transversal.  There exists a sequence $(\varepsilon_n)$ decreasing towards zero, and, for every~$n$, a line $(\ell_n)$ transversal to~$\G^{\varepsilon_n}$: it intersects $H^{\varepsilon_n}$ ($H\in\G$) at point $a_{H,n}$.  Up to taking a subsequence, we can assume that $(\ell_n)$ converges towards a line~$\ell$, and that each sequence $(a_{H,n})$ converges towards some point~$a_H$ (by compactness of~$\RG_{2,d+1}$, and since the objects are bounded).  Of course, each $a_H$ belongs to~$\ell$, and also to the closure of each $H^{\varepsilon_n}$, hence to~$H$, since $H$ is closed.  So $\G$ has a line transversal.
\end{proof}

\subsection{Three theorems in geometric transversal theory}

We can now deduce three transversal Helly numbers from our main result. The main interest in these derivations is not that the bounds are better; in fact, one matches the previously known bound, one is weaker ($10$ instead of $5$), and the last one is better ($4d-2$ instead of $4d-1$, when $d \ge 6$). They do show, however, that the combinatorial and homological conditions of Theorem~\ref{thm:non-uniform} may be useful in identifying situations where the transversal Helly numbers are bounded; in fact the question whether our second and third examples afford bounded transversal Helly numbers were raised in the late 1950's and only answered in 1986 and 2006. Refer to Table~\ref{tab:parameters} for a summary of the parameters used in the applications of Theorem~\ref{thm:non-uniform}.

\begin{table}\label{t:table}
\begin{center}
\begin{tabular}{|p{.3\linewidth}||c||c||c|c|c|c|}
\hline
Shape & Previous bound & Our bound & $d_\Gamma$ & $s$ & $t$ & $r$ \\
\hline \hline
Parallelotopes in $\R^d$ ($d\ge2$) & $2^{d-1}(2d-1)$ \cite{s-tscpap-40} & $2^{d-1}(2d-1)$ & $2d-2$ & $d+1$ & $1$ & $2^{d-1}$ \\
\hline
Disjoint translates of a planar convex figure & $5$ \cite{t-pgcct-89} & $10$ & $2$ & $3$ & $4$ & $2$\\
\hline
Disjoint unit balls in $\R^d$: & & & & & &\\
\hfill $d=2$ & 5 \cite{Danzer57} & 12 & $2d-2$ & $d+1$ & $1$ & $3$\\
\hfill $d=3$ & 11 \cite{cghp-hhtdus-08} & $15$ & $2d-2$ & $d+1$ & $1$ & $3$\\
\hfill $d=4$ & 15 \cite{cghp-hhtdus-08} & $20$ & $2d-2$ & $d+1$ & $9$ & $2$\\
\hfill $d=5$ & 19 \cite{cghp-hhtdus-08} & $20$ & $2d-2$ & $d+1$ & $9$ & $2$\\
\hfill $d\ge6$ & $4d-1$ \cite{cghp-hhtdus-08} & $4d-2$ & $2d-2$ & $d+1$ & $9$ & $2$\\
\hline
\end{tabular}

\caption{Parameters used to derive bounds on transversal Helly numbers from Theorem~\ref{thm:non-uniform}.} \label{tab:parameters}
\end{center}
\end{table}

\paragraph{Parallelotopes in arbitrary dimension.}
Let $\h$ be a finite family of parallelotopes in $\R^d$ with edges parallel to the coordinate axis. Santal\'o~\cite{s-tscpap-40} showed that the transversal Helly number $\tau_1$ of $\h$ is at most $2^{d-1}(2d-1)$. Here is how Santal\'o's theorem can be seen to follow from Theorem~\ref{thm:non-uniform}.  We can restrict ourselves to \emph{open} parallelotopes by Lemma~\ref{lem:transcompact}.

Let $D$ be the set of directions in $\R\p^{d-1}$ that are \emph{not} orthogonal to the direction of any coordinate axis.  $D$ has exactly $2^{d-1}$ connected components.  Recall that $p\inv(D)$ is the set of lines whose direction is in~$D$.  When studying the existence of transversals to~$\h$, it does not harm to restrict to lines in~$p\inv(D)$, since the set of transversals to~$\h$ is open and since the complement of~$p\inv(D)$ has empty interior.

For each connected component of~$D$, the set of transversals to $\h$ with direction in this component can be seen to be homeomorphic to the interior of a polytope in a $(2d-2)$-dimensional affine subspace of $\R^{2d}$ by adequate use of \emph{Cremona coordinates}~\cite{ghprs-ccfcts-06}. In particular, for any $\G\subseteq \h$, the set $T(\G)\cap p\inv(D)$ consists of at most $2^{d-1}$ contractible components. Moreover, if $\Gamma=p\inv(D)$, then $d_\Gamma\le 2d-2$ by Lemma~\ref{lem:homdim}.  Theorem~\ref{thm:mainresult} now implies an upper bound of $2^{d-1}(2d-1)$ in the Helly number of transversals of parallelotopes.

If we consider the partition of line space into $2^{d-1}$ regions $R_1, \ldots, R_{2^{d-1}}$ induced by the above partition of $\R\p^{d-1}$, the Cremona coordinates recast the set of line transversals in each $R_i$ into a convex set, and Santal\'o's theorem follows directly from applying Helly's theorem inside each $R_i$~\cite{ghprs-ccfcts-06}. While this is simpler, we know of no other example where a transversal Helly number is obtained by partitioning the space of lines and identifying a convexity structure in each region. In fact, the definition of convexity structures on the Grassmannian in itself raises several issues~\cite{gp-ftcagm-95}.

\paragraph{Disjoint translates in the plane.}
Tverberg~\cite{t-pgcct-89} showed that for any compact convex subset $D \subset \R^2$ with non-empty interior, the transversal Helly number $\tau_1$ of any finite family $\h$ of disjoint translates of $D$ is at most $5$. This settled a conjecture of Gr\"unbaum~\cite{g-ct-58} previously proven in the cases where $D$ is a disk~\cite{Danzer57} and a square~\cite{g-ct-58}, or with the weaker bound of $128$~\cite{katch86}. Tverberg's proof uses in an essential way properties of geometric permutations of collections of disjoint translates of a convex figure~\cite{kll-gpdt-87}. Here, we show how an upper bound of $10$ can be easily derived from Theorem~\ref{thm:non-uniform} and the sole property that the number of geometric permutations of $n$ disjoint translates of a compact convex set with non-empty interior in $\R^2$ is at most $3$ in general and at most $2$ if $n \ge 4$~\cite{kll-gpdt-87}.

First, remark that instead of translates of a compact convex set, we can consider translates of an \emph{open} convex set (using Lemma~\ref{lem:transcompact}, by letting $H^\varepsilon$ be the set of points at distance strictly less than~$\varepsilon$ from~$H$). Now, observe that for any $A_i \in \h$ the set $T_i = T(\{A_i\})$ has the homotopy type of~$\R\p^1$. Moreover, for any sub-family $\G \subseteq \h$ of size at least two, the set of directions in $\K(\G)$ corresponding to a given geometric permutation of $\G$ is a connected proper subset of~$\R\p^1$, and Lemma~\ref{lem:directions} implies that $T(\G)$ is acyclic with slack $s=3$. Moreover, the number of components in $T(\G)$ is at most the maximum number of geometric permutations of $\G$, that is at most $3$ in general and at most $2$ when $|\G| \ge 4$~\cite{kll-gpdt-87}. We can therefore apply Theorem~\ref{thm:non-uniform} with $d_\Gamma=2$, $s=3$, $t=1$ and $r=3$, getting an upper bound of $12$, or with $d_\Gamma=2$, $s=3$, $t=4$ and $r=2$, obtaining the better bound of $10$.

In dimension $3$ or more there exist families of disjoint translates of a polyhedron with arbitrarily many connected components of line transversals; in other words, $r$ cannot be bounded. In that setting, indeed, Tverberg's theorem is known not to generalize~\cite{hm-nhtst-04}.

\paragraph{Disjoint unit balls in arbitrary dimension.}
Cheong et al.~\cite{cghp-hhtdus-08} showed that the transversal Helly number $\tau_1$ of any finite collection $\h$ of disjoint equal-radius closed balls in $\R^d$ is at most $4d-1$. That this number is bounded was first conjectured by Danzer~\cite{Danzer57} and previously proven for $d=2$~\cite{Danzer57} and $d=3$~\cite{hkl-httlt-03} or under various stronger assumptions (see~\cite{cghp-hhtdus-08} and the discussion therein). The proof of Cheong et al.~\cite{cghp-hhtdus-08} combines a characterization of families of geometric permutations of $n \ge 9$ disjoint unit balls with a local application of Helly's topological theorem. Here we show how Theorem~\ref{thm:non-uniform} and some ingredients of their proofs yield a slightly improved bound.

First, note that by Lemma~\ref{lem:transcompact}, we can consider open balls with the same radius (say one).  Observe that for any $A_i \in \h$ the set $T_i = T(\{A_i\})$ has the homotopy type of~$\R\p^{d-1}$, and is therefore homologically trivial in dimension $d$ and~higher. Then, for any sub-family $\G \subseteq \h$ of size at least two, the set of directions in $\K(\G)$ corresponding to a given geometric permutation of $\G$ is convex\footnote{Convexity in $\R\p^{d-1}$ is relative to the metric induced through the identification $\R\p^{d-1} = \s^{d-1}/\Z_2$.}~\cite{bgp-ltdb-08} and therefore contractible. In other words, $\K(\G)$ is a disjoint union of contractible sets; so is $T(\G)$ by Lemma~\ref{lem:directions}.  It follows that for any $\G \subseteq \h$, $T(\G)$ is acyclic with slack $d+1$. Moreover, for any $d$ the number of geometric permutations of a family of $n$ disjoint equal-radius balls in $\R^d$ is at most $3$ in general and at most $2$ when $n \ge 9$~\cite{cgn-gpdus-05}. We can thus apply Theorem~\ref{thm:non-uniform} with $d_\Gamma = 2d-2$, $s=d+1$, $t=9$, and $r=2$, obtaining the upper bound of $2 \max(2d-1,10)$. For $d \ge 6$, this yields the upper bound of $4d-2$, but for $d \in \{2,3,4,5\}$ this bound is only $20$.  In the case $d=2$ (resp.\ $d=3$) it can be improved to $12$ (resp.~$15$) by using $d_\Gamma = 2d-2$, $s=d+1$, $t=1$, and $r=3$.

It is conjectured that any family of $4$ or more disjoint equal-radius balls in $\R^d$ has at most two geometric permutations. If this is true, then our bounds would improve to $4d-2$ for any $d \ge3$. Since the transversal Helly number $\tau_1$ of disjoint equal-radius balls is at least $2d-1$~\cite{lowerbounds}, this number is known up to a factor of $2$. Families of $n$ disjoint balls with arbitrary radii in $\R^d$ have up to $\Theta(n^{d-1})$ geometric permutations~\cite{sms-sbgpp-00} and their transversal Helly number is unbounded; if the radii are required to be in some fixed interval $[1, \rho]$, this bound reduces to $O(\rho^{\log \rho})$~\cite{zs-gpbbs-03} and Theorem~\ref{thm:non-uniform} similarly implies that the first transversal Helly number is $O(d\rho^{\log \rho})$, where the constant in the $O()$ is independent of $\rho$, $n$ and $d$. 

\appendix

\section{Elementary properties of spectral sequences}
\label{A:specseq}
In this appendix, we quickly recall some standard elementary properties of
the theory of spectral sequences (in homological algebra) for the
convenience of the non-experts. All  statements below can be
found in any standard reference such as~\cite{Weibel, McCl, Spa}.

In this paper, we use spectral sequences as a tool to compute or
approximate some homology groups, and we consider only spectral sequences
of \emph{homological type} and lying in the \emph{first quadrant}. Further
we are working over a ground \emph{field} (which is indeed
$\mathbb{Q}$). In this context, a spectral sequence is a sequence
$(E^{r}_{\bullet,\bullet})_{r\geq r_0}$ (where $r_0\geq 0$ is an integer)
of bigraded abelian groups $E_{p,q}^r$ (where $p,q\geq 0$ are integers by
our first quadrant assumption) equipped with a differential
$d^r:E^r_{p,q}\to E^r_{p-r, q+r-1}$ (in particular $d^r\circ d^r=0$) which are required to satisfy the following
assumption:   the groups in~$E^{r+1}_{\bullet,\bullet}$ are the homology groups of
  the chain complex appearing in~$E^r_{\bullet,\bullet}$; namely,
  $E^{r+1}_{p,q}$ is the quotient of the kernel of $d^r:E^r_{p,q}\to
  E^r_{p-r,q+r-1}$ by the image of $d^r:E^r_{p+r,q-r+1}\to
  E^r_{p,q}$. In other words, the groups $E^r_{\bullet,\bullet}$ can be computed inductively from $E^{r_0}_{\bullet,\bullet}$
 by taking the homology of  $E^{r'}_{\bullet,\bullet}$ ($r_0\leq r' <r$) at each intermediate step.
The term $E^r_{\bullet,\bullet}$ is sometimes called the \emph{$r$-page} of
the spectral sequence. Note that, for degree reasons, $d^{r}:E^r_{p,q}\to
E^r_{p-r, q+r-1}=0$ for $r>p$.  It follows that the terms $E^r_{p,q}$
stabilizes, that is $E^{r}_{p,q}\cong E^{r+1}_{p,q}\cong \cdots$ for $r$
large enough. We write $E^\infty_{p,q}$ for the stabilized groups
$E^\infty_{p,q} \cong \cdots \cong E^{p+2}_{p,q}\cong E^{p+1}_{p,q}$.

If $H_\bullet$ are some homology groups and
$(E^{r}_{\bullet,\bullet})_{r\geq r_0}$ is a spectral sequence, one says
that $E^r_{p,q}$ \emph{converges to}\footnote{Since we are only interested
  in the dimension of homology groups over a field, we do not need to care
  about filtrations issues.} $H_\bullet$ (which is denoted by $E^{r}_{p,q}
\Longrightarrow H_{p+q}$) if there is a linear isomorphism of vector spaces
$H_{n}\cong \bigoplus_{p+q=n} E^\infty_{p,q}$ for every $n\geq 0$.

\smallskip

The main result we really need about the theory of spectral sequence is the
following (almost tautological) property: let $E_{p,q}^r$ ($r\geq r_0$,
$p,q\geq 0$) be a homology spectral sequence converging to $H_{p+q}$.
Assume that there exists integers $N$ and $r\geq r_0$ such that
$E^r_{p,q}=0$ whenever $p+q \geq N$.  Then $H_k=0$ for $k\geq N$ as well.

Standard examples of spectral sequences arise from filtrations on chain
complexes. For instance, if $(C_\bullet,d)$ is a chain complex equipped
with a filtration $\{0\}\subset F_0(C_\bullet)\subset F_1(C_\bullet)\subset
\cdots$, there is a spectral sequence converging to
$H_\bullet(C_\bullet,d)$ whose first page is $E^0_{p,q}:=
F_p(C_{p+q})\slash F_{p-1}(C_{p+q})$. There are similar filtrations for
(nice enough, for instance CW-complexes) topological spaces inducing
spectral sequences in (singular or CW) homology.

An important class of examples is given by bicomplexes. A bicomplex is a
bigraded vector spaces $(C_{p,q})_{p,q\geq 0}$ equipped with two
differentials $d_h: C_{p,q}\to C_{p-1,q}$ and $d_v:C_{p,q}\to C_{p,q-1}$
such that $d_v\circ d_h=-d_h\circ d_v$. We will sometimes refer to $d_v$ as
the \emph{vertical} differential and $d_h$ as the \emph{horizontal}
differential.  The (total) homology of $(C_{p,q}, d_h, d_v)_{p,q\geq 0}$ is
the homology of the (total) complex $(\bigoplus_{p+q=n}C_{p,q},
d_v+d_h)_{n\geq 0}$. In that case, there is a spectral sequence $E^r_{p,q}$
converging to the homology of the total complex whose page $E^1_{p,q}$ is
isomorphic to $E^1_{p,q}= H_q(C_{p,\bullet}, d_v)$ and whose differential
$d^1:E^1_{p,q}\to E^1_{p-1,q}$ is given by $d_h$ (or more precisely by
$H_q(d_h):H_q(C_{p,\bullet})\to H_q(C_{p-1,\bullet})$). There is also a
similar spectral sequence exchanging the roles of $p$ and $q$.

\section{Homology of spaces with contractible fibers}
\label{A:contractible}

In some situations, topological (or homological) properties of a topological space $X$ can be understood by considering a projection $p:X \to Y$ with \emph{contractible} fibers. An example from the geometric transversal literature is when $X$ is the set of line transversals to some family of convex sets and $p$ maps a line to its direction. While simple settings allow for elementary proofs (see e.g.\ the proof of~\cite[Lemma~14]{cghp-hhtdus-08}), standard arguments in algebraic topology lead to more general statements such as Lemma~\ref{lem:directions}. In this appendix, we collect some of these arguments in the hope that they can be useful in other contexts.

\begin{lemma}\label{lem:proj-trivial}
  Let $\pi:X\to Y$ be a continuous surjective map from a topological space~$X$ onto a
  topological space~$Y$. We assume that the fiber $\pi\inv(y)$ is contractible for every $y\in
  Y$. Assume \emph{either one} of the following assumptions is satisfied
\begin{enumerate}
 \item  $X$, $Y$ are paracompact Hausdorff and, further, $\pi$ is closed;
\item $X$ and  $Y$  are manifolds and $\pi$ is a submersion;
 \item $X$ and $Y$ are (the geometric realization of) simplicial complexes and $\pi: X\to Y$ is simplicial;
\item $\pi: X\to Y$ is a fibration;
\item $X=\bigcup_{n\geq 0} X_n$ is a union of closed subsets (with $X_n$ in the relative interior of $X_{n+1}$) such that $\pi_{|X_n}:X_n\to Y$ is proper with contractible fibers.  
\end{enumerate}
Then, the natural map $\pi_*:H_n(X)\to H_n(Y)$ is an isomorphism for all $n$. 

Further, if $X$ and $Y$ are CW-complexes, then $\pi$ 
is an homotopy equivalence when either assumption 3. or 4. is satisfied.
\end{lemma}
\begin{proof}
  Let us recall that we work over a characteristic zero field and thus it is equivalent to prove the result in cohomology by the universal coefficient theorem~\cite{Spa, g-lat-67}. The case of assumption 1. reduces to the Vietoris-Begle mapping theorem (see~\cite[Theorem 15, Section 6.9]{Spa}). The case of assumption 2. is the main result of~\cite{Sm-sub-77}. The case of assumption 3. (as well as its homotopic version) is proved in~\cite{GaVo}.  The case of assumption 5. is a corollary of~\cite[Proposition 2.7.8]{KaSc} applied to a constant sheaf. When $\pi: X\to Y$ is a fibration with contractible fibers, it follows from the long exact sequence of homotopy groups of a fibration (see~\cite{Spa} or any textbook in algebraic topology) that the induced maps $\pi_*: \pi_k(X,{x_0})\to \pi_k(Y,{y_0})$ is an isomorphism for any $k$ and choice of a base point $x_0\in X$ (recall that we assume $\pi$ to be surjective). Thus $\pi: X\to Y$ is a weak homotopy equivalence and thus induces an isomorphism in (co)homology~\cite[Theorem 25, Section 7.6]{Spa}. Further, weak homotopy equivalences between CW-complexes are homotopy equivalences (see~\cite[Section 7.6]{Spa}).
\end{proof}

\noindent
Although some spaces satisfy several of the assumptions 1. to 5. simultaneously, these assumptions are not equivalent in general; any of them is enough to ensure the result.  Let us give some examples in which Lemma~\ref{lem:proj-trivial} applies. 
\begin{itemize}
 \item If $X$ is (Hausdorff) compact and $Y$ is Hausdorff, then Assumption 1. is automatically satisfied.
\item Recall that a large class of examples of fibrations are given by fiber bundles~\cite{Spa}. We recall that  $\pi: X\to Y$ is a fiber bundle if there exists a topological space $F$ (the fiber) such that any point in $Y$ has a neighborhood $U$ such that $\pi\inv(U)$ is homeomorphic to a product $U\times \pi\inv(y)$ in such a way that the map $\pi_{|\pi\inv(U)}$ identifies with the first projection $U\times \pi\inv(y)\to U$. That is, the map $\pi: X\to Y$ is locally trivial with fiber homeorphic to $F$. In particular, covering spaces,
vector bundles, principal group bundles are fibrations.
\item If $X$ (Hausdorff) can be covered by an union $\bigcup X_n$ of compact spaces such that the fibers of $p_{| X_n}$ are contractible, then 5. is satisfied and the result of the lemma holds.
\end{itemize}

\bibliography{trans,add}
\bibliographystyle{plain}
\end{document}